\documentclass[12pt, reqno]{amsart}
\usepackage[a4paper,margin=1.15in]{geometry}
\usepackage[hidelinks]{hyperref}
\hypersetup{
pdfauthor={Xiang Fang, Feng Guo, Aman Mishra, and P. Muthukumar},
pdftitle={Critical Norm Profiles for Finite-Prime Composition Operators
on the Hardy Space of Dirichlet Series},
colorlinks=true,linkcolor=blue
}

\usepackage[backend=biber, style=ieee, sorting=nyt]{biblatex}
\usepackage{amsthm,amsmath,amssymb,mathtools}
\numberwithin{equation}{section}
\usepackage[shortlabels]{enumitem}
\usepackage{microtype}
\usepackage{caption}

\usepackage{romannum}
\usepackage{xcolor}
\usepackage{pgfplots}
\usetikzlibrary{pgfplots.fillbetween}
\pgfplotsset{compat=1.18}
\usepackage[most]{tcolorbox}
\tcbset{
  enhanced,
  breakable,
  boxrule=0.25mm,
  left=6pt,right=6pt,top=1pt,bottom=1pt,
  arc=0mm
}

\newtheorem{theorem}{Theorem}[section]
\newtheorem{lemma}[theorem]{Lemma}
\newtheorem{proposition}[theorem]{Proposition}
\newtheorem{corollary}[theorem]{Corollary}

\begin{document}
\pagenumbering{arabic}
\setcounter{page}{1}
\title[Critical norm profiles for finite-prime composition operators]{Critical Norm Profiles for Finite-Prime Composition Operators on the Hardy Space of Dirichlet Series}

\author[X. Fang]{Xiang Fang}
\address{Xiang Fang, Department of Applied Mathematics, National Yang Ming Chiao Tung University, Hsinchu, Taiwan (R.O.C.).}
\email{\textcolor[rgb]{0.00,0.00,0.84}{xfang@nycu.edu.tw}}

\author[F. Guo]{Feng Guo}
\address{Feng Guo, School of Mathematics, Nanjing University of Aeronautics and Astronautics, Nanjing 210016, P. R. China.}
\email{\textcolor[rgb]{0.00,0.00,0.84}{70207994@nuaa.edu.cn}}

\author[A. Mishra]{Aman Mishra}
\address{Aman Mishra, Department of Mathematics, Indian Institute of Technology, Kanpur - 208016, India.}
\email{\textcolor[rgb]{0.00,0.00,0.84}{aamanmishra121@gmail.com, amanr24@iitk.ac.in}}

\author[P. Muthukumar]{P. Muthukumar}
\address{P. Muthukumar, Department of Mathematics, Indian Institute of Technology, Kanpur - 208016, India.}
\email{\textcolor[rgb]{0.00,0.00,0.84}{pmuthumaths@gmail.com, muthu@iitk.ac.in}}



\begin{abstract}
We identify the critical boundary operator-norm profile of finite-prime composition operators on the Hardy--Hilbert space \(\mathcal H^2\) of Dirichlet series. For
\[
\varphi_{\delta,\boldsymbol\rho}(s)
=
\frac12+\delta
+
\delta\sum_{j=1}^d\rho_jp_j^{-s},
\qquad
\boldsymbol\rho\in B_d,
\]
the renormalized positive coefficient operators converge uniformly in operator norm, with \(O(\delta)\) error, to an explicit multivariate weighted Hankel operator \(\mathcal H_{\boldsymbol\rho}\); consequently,
\[
2\delta\|C_{\varphi_{\delta,\boldsymbol\rho}}\|^2
=
\|\mathcal H_{\boldsymbol\rho}\|
+
O(\delta)
\]
uniformly over \(B_d\). We show that the limiting operator admits the total-degree reduction
\[
\mathcal H_{\boldsymbol\rho}
\simeq
D_{\boldsymbol\rho}
H_{R_{\boldsymbol\rho}/2}
D_{\boldsymbol\rho}\oplus\mathbf{0},
\]
where the diagonal factors are convolution-collision norms of the normalized prime weights. This structure, together with the affine comparison principle of Brevig and Perfekt, yields an explicit concentration inequality for \(\|\mathcal H_{\boldsymbol\rho}\|\), identifies the one-prime configurations as the exact equality cases in the limiting norm estimate, and gives a quantitative deficit away from them. For fixed \(\sigma>\frac12\),  we also obtain a second-order expansion of the squared
norm and fully finite-dimensional approximations with explicit total-degree and Dirichlet-sum truncation errors. Together, these results show that a single coefficient-operator structure governs the singular boundary profile, the fixed-\(\sigma\) perturbative regime, and certified finite-dimensional approximation.
\end{abstract}

\maketitle
\vspace{-1cm}
\tableofcontents

\section{Introduction}

\subsection{The norm problem and the finite-prime model}

The Hardy--Hilbert space of Dirichlet series is
\[
\mathcal H^2
=
\left\{
f(s)=\sum_{n\geq1}a_n n^{-s}:
\sum_{n\geq1}|a_n|^2<\infty
\right\},
\]
equipped with its coefficient norm. It is a reproducing-kernel Hilbert
space on the half-plane
\[
\mathbb C_{1/2}
=
\{s\in\mathbb C:\operatorname{Re}s>1/2\},
\]
with reproducing kernel determined by the Riemann zeta function. The interaction between the Hilbert-space structure of $\mathcal H^2$
and unique prime factorization of natural numbers makes composition on $\mathcal H^2$
markedly different from composition on the classical one-variable
Hardy spaces: the symbol controls both analytic mapping and the arithmetic
organization of the coefficients.

The foundational work of Hedenmalm, Lindqvist, and Seip established the
basic function-space theory of $\mathcal H^2$, while Gordon and
Hedenmalm characterized the analytic symbols that induce bounded
composition operators on $\mathcal H^2$; see \cite{hls,GH,quef}. Given
such a symbol $\varphi:\mathbb C_{1/2} \to \mathbb C_{1/2}$, the associated composition operator $C_\varphi$ is defined by
\[
C_\varphi f=f\circ\varphi,
\qquad f\in\mathcal H^2.
\]
Once boundedness is understood, a finer
operator-theoretic problem remains: determine, or at least describe
sharply, the norm of $C_\varphi$.  In the zero-characteristic case this norm is sensitive not only to the range of
the symbol but also to the distribution of its Dirichlet coefficients.
The boundedness classification therefore provides the admissible class,
but it does not determine the norm geometry within that class.

In this paper, we consider finite-prime  symbols
\[
\varphi(s)
=
\sigma+\sum_{j=1}^d r_jp_j^{-s},
\]
where $p_1,\ldots,p_d$ are distinct primes. Although the symbol involves only finitely many primes,
the associated composition operator acts on the infinite-dimensional
space $\mathcal H^2$. The distinct primes lead naturally to independent
multi-index coordinates, and the image of a Dirichlet polynomial gives
rise to a coefficient operator from
$\ell^2(\mathbb N)$ to $\ell^2(\mathbb N_0^d)$. This finite-prime model
provides a convenient setting in which the distribution of the
coefficients can be studied explicitly.

We study three regimes of the same norm problem.  In the critical regime,
\(\sigma=\frac12+\delta\) and \(r_j=\delta\rho_j\), with the normalized
direction \(\boldsymbol\rho\) fixed as \(\delta\downarrow0\).  In the
perturbative regime, \(\sigma>\frac12\) is fixed and
\(R=\sum_jr_j\downarrow0\).  Finally, we develop finite-dimensional
approximations of the norm from the same  coefficient operator.  The central
question is whether the singularly normalized critical family has an
operator-norm limit and, if so, how that limit records the distribution of
mass among the prime coordinates.  This asks for more than the leading
growth rate of a scalar norm: it asks for the positive operator that
survives after renormalization.  The fixed-\(\sigma\) expansion and
finite-section theorem arise from the same coefficient realization, so the
three regimes test the same operator model at a boundary scale, at an
interior perturbative scale, and through effective finite-dimensional
truncation.
\subsection{Earlier work and the precise gap}

Composition operators on $\mathcal H^2$ have been studied from several
complementary directions. Bayart initiated a systematic study of the
zero-characteristic case, and subsequent work has addressed compactness,
approximation numbers, mean counting functions, Schatten classes, and
related questions; see, for example,
\cite{Ba,Ba2,CQV,BP2,AP,AK,BA}. For operator norms, Brevig obtained
sharp estimates in a substantial range, with a subsequent corrigendum,
while Muthukumar, Ponnusamy, and Queff\'elec established upper and lower
bounds for the one-prime family; see
\cite{Brevig,Brevig2023,MPQ}. These results show that the
norm problem has a developed quantitative theory, even though exact
formulas remain unavailable for broad zero-characteristic families.  They
also indicate why asymptotic and comparison questions are natural: the
reproducing-kernel bounds detect the boundary singularity, while the
coefficient distribution is not visible from the center and total radius
alone.  The present paper is therefore positioned within an existing norm
theory, rather than as the first investigation of these operators.

The closest predecessor to the present work is the affine
subordination principle of Brevig and Perfekt \cite{BP1}. For the
finite-prime symbol
\[
\varphi(s)
=
\sigma+\sum_{j=1}^d r_jp_j^{-s},
\]
they proved
\[
\|C_\varphi\|
\leq
\|C_\psi\|,
\]
where
\[
\psi(s)
=
\sigma+
\left(\sum_{j=1}^d r_j\right)2^{-s}.
\]
Thus, at fixed center $\sigma$ and fixed total coefficient mass
$R=\sum_{j=1}^d r_j,$
concentration on a single prime is extremal for the operator norm. In particular,
one-prime coefficient vectors, up to the choice of prime, are the
maximizing configurations in that fixed-parameter problem. A central question in the present paper is how this concentration phenomenon is reflected in the asymptotic operator structure near the critical boundary.

What is missing from the fixed-parameter theory is the singular operator
produced when the center approaches the boundary and all coefficients
vanish on the same scale.  Scalar subordination does not identify that
operator, establish uniform operator-norm convergence to it, or reveal how
its multivariate coefficient matrix decomposes.  A related orthogonal
decomposition was developed by Brevig and Perfekt \cite{BP2022} for
positive-characteristic symbols \(c_0s+\varphi_0(s)\), organized by
multiplicative prime support and used for approximation numbers and
compactness.  Our setting is different: the characteristic is zero, the
boundary limit is singular, and the relevant decomposition is by total
degree in the limiting operator.  The purpose of the paper is to
identify this critical operator profile and analyze the geometry it
retains.  In particular, we seek an operator-level limit that is uniform
over the normalized coefficient simplex, a structural reduction that makes
its norm accessible, and a precise account of how the fixed-parameter
concentration principle appears in the limit.

\subsection{Main results and their interpretation}

We now describe the main results. In the critical regime, set
\[
\sigma=\frac12+\delta,
\qquad
r_j=\delta\rho_j,
\]
and let
\[
B_d
=
\left\{
\boldsymbol\rho=(\rho_1,\ldots,\rho_d)\in[0,1]^d:
\rho_1+\cdots+\rho_d\leq1
\right\}.
\]
Thus
\[
\varphi_{\delta,\boldsymbol\rho}(s)
=
\frac12+\delta
+
\delta\sum_{j=1}^d\rho_jp_j^{-s}.
\]
Let $T_{\delta,\boldsymbol\rho}$ denote the corresponding coefficient
operator and put
\[
A_{\delta,\boldsymbol\rho}
=
T_{\delta,\boldsymbol\rho}
T_{\delta,\boldsymbol\rho}^*.
\]
The scaling by $\delta=\sigma-\frac12$ isolates the parameter governing
the approach to the critical boundary while keeping the normalized
coefficient vector fixed.

Our first main result, Theorem~\ref{main}, establishes a uniform operator-norm approximation to the limiting positive operator $\mathcal H_{\boldsymbol\rho}$. More precisely, as $\delta\downarrow0$,
\[
\sup_{\boldsymbol\rho\in B_d}
\left\|
2\delta A_{\delta,\boldsymbol\rho}
-
\mathcal H_{\boldsymbol\rho}
\right\|
=
O(\delta),
\]
and consequently
\[
\sup_{\boldsymbol\rho\in B_d}
\left|
2\delta
\|C_{\varphi_{\delta,\boldsymbol\rho}}\|^2
-
\|\mathcal H_{\boldsymbol\rho}\|
\right|
=
O(\delta).
\]
Thus the critical regime admits a uniform operator-level limit, and the
renormalized norm is governed by the norm of a concrete limiting
 operator $\mathcal H_{\boldsymbol\rho}$.  This is the principal asymptotic statement of the paper: it identifies the entire critical profile, rather than only the value of its norm along selected directions.
 The limiting norm satisfies
\[
\|\mathcal H_{\boldsymbol\rho}\|
\le
\frac{2}{1+\sqrt{1-R_{\boldsymbol\rho}^2}}.
\]
Using the affine subordination principle of Brevig and Perfekt together
with the critical asymptotic above, we obtain a quantitative lower bound
for the gap in this estimate and characterize its equality cases: equality
holds if and only if at most one component of
$\boldsymbol\rho$ is nonzero.  In this way, the critical limit retains a quantitative form of the concentration
phenomenon underlying affine subordination.

We then turn to the structure of the limiting operator $\mathcal H_{\boldsymbol\rho}$ which is described in
Theorem~\ref{thm:boundary-Hankel-reduction}. Writing
$R_{\boldsymbol\rho}
= \rho_1+\cdots+\rho_d,$
we show that, when $R_{\boldsymbol\rho}>0$, the multivariate operator
$\mathcal H_{\boldsymbol\rho}$ admits a decomposition, up to unitary
equivalence and a diagonal transformation, of the form
\[
\mathcal H_{\boldsymbol\rho}
\simeq
D_{\boldsymbol\rho}
H_{R_{\boldsymbol\rho}/2}
D_{\boldsymbol\rho}
\oplus\mathbf 0,
\]
where the weighted Hankel operator
$H_{R_{\boldsymbol\rho}/2}$ is studied in Section~\ref{Sec:2}. The
reduction is organized by total degree and converts the multivariate
critical problem into a one-variable weighted Hankel model while
retaining information about the distribution of
$\boldsymbol\rho$ among the prime coordinates. The diagonal coefficients of
$D_{\boldsymbol\rho}$ admit a convolution-collision representation in
terms of the normalized prime weights. In the genuinely multivariate
case these coefficients tend to zero, and hence
$\mathcal H_{\boldsymbol\rho}$ is compact.

The same coefficient realization also yields a regular perturbative
result away from the boundary. For fixed $\sigma>\frac12$ and
\[
R=\sum_{j=1}^d r_j\downarrow0,
\]
Theorem~\ref{T:finite-prime-fixed-sigma-expansion} gives
\[
\left\|
C_{\sigma+\sum_{j=1}^d r_jp_j^{-s}}
\right\|^2
=
\zeta(2\sigma)
+
\frac{\zeta'(2\sigma)^2}{\zeta(2\sigma)}
\sum_{j=1}^d r_j^2
+
O(R^4).
\]
Thus the first nonconstant term depends on the quadratic mass
\[
\sum_{j=1}^d r_j^2,
\]
rather than only on the total mass $R$. The norm therefore detects the
distribution of the prime coefficients already at second order. This
regular perturbative expansion complements the singular critical
asymptotics by describing the small-coefficient behavior at a fixed
interior point of the half-plane.

Finally, Section~\ref{sec:5} develops finite-dimensional approximations
of the norm with explicit total error bounds.  Let
$\lambda_{N,M}(\sigma,\boldsymbol r)$ denote the resulting
finite-dimensional approximation. Theorem~\ref{thm:full-error-estimate}
shows that, whenever
\[
2R<\varepsilon<2\sigma-1,
\]
\[
0
\leq
\|C_{\varphi}\|^2
-
\lambda_{N,M}(\sigma,\boldsymbol r)
\leq
E_N(\sigma,\boldsymbol r,\varepsilon)
+
\eta_{N,M}(\sigma,\boldsymbol r,\varepsilon),
\]
where
\[
E_N(\sigma,\boldsymbol r,\varepsilon)
=
\zeta(2\sigma-\varepsilon)
\frac{(4R^2/\varepsilon^2)^N}
{1-4R^2/\varepsilon^2},
\]
and
\[
\eta_{N,M}(\sigma,\boldsymbol r,\varepsilon)
=
\left(
\sum_{j=0}^{N-1}
\frac{(2j)!}{\varepsilon^{2j}}
a_j(\boldsymbol r)^2
\right)
\frac{M^{-(2\sigma-\varepsilon-1)}}
{2\sigma-\varepsilon-1}.
\]
Thus, the norm of $C_{\varphi}$ can be
approximated by finite-dimensional matrices with a fully explicit
and rigorous error bound.

\subsection{Proof strategy, scope, and organization}

The proof begins with the coefficient realization of $C_\varphi$ and the
associated positive operator
$T_{\sigma,\boldsymbol r}T_{\sigma,\boldsymbol r}^*$, whose matrix
entries are expressed in terms of derivatives of the zeta function.
Under the critical scaling, the pole of the zeta function determines the
leading term $\mathcal H_{\boldsymbol\rho}$, while the remainder is
controlled uniformly by Cauchy estimates and Hilbert--Schmidt norm
bounds. This yields the uniform operator-norm approximation in
Theorem~\ref{main}. The Brevig--Perfekt affine subordination principle,
combined with this critical asymptotic and the explicit one-prime
limit, then gives the quantitative gap estimate and the complete
characterization of the equality cases. We subsequently analyze the
structure of $\mathcal H_{\boldsymbol\rho}$ by decomposing it according
to total degree, which leads to the weighted Hankel model and the
convolution representation of the diagonal coefficients.

For the fixed-$\sigma$ regime, the same coefficient realization gives an
absolutely convergent rank-one expansion of
$T_{\sigma,\boldsymbol r}^*T_{\sigma,\boldsymbol r}$. Separating the
degree-zero and degree-one contributions yields the second-order
perturbative expansion, while the remaining degrees are controlled by a
geometric majorant, giving the $O(R^4)$ remainder. The same majorant,
applied to the tail beyond a prescribed total degree, yields the
certified finite-section estimate. A further truncation of the
Dirichlet-series sums gives the fully finite approximation and its
explicit total error bound.

Our results concern finite-prime, zero-characteristic symbols with fixed
finite $d$. In the critical regime, the normalized direction
$\boldsymbol\rho$ is fixed along each family, although the estimates are
uniform over $B_d$; in the perturbative regime, $\sigma>\frac12$ is
fixed. We do not obtain a closed formula for the norm of a general
finite-prime symbol, nor do we treat arbitrary symbols in the
Gordon--Hedenmalm class. Regimes in which the number of prime variables
grows, or in which the coefficients approach the critical boundary on
different asymptotic scales, are outside the scope of the present work.

This paper is organized as follows. Section~\ref{Sec:2} establishes the
weighted Hankel benchmark. Section~\ref{sec:3} develops the coefficient
realization, the uniform critical operator limit, the limiting
concentration inequality, the total-degree reduction, and the collision
profile. Section~\ref{sec:4} proves the fixed-\(\sigma\) squared-norm
expansion. Finally, Section~\ref{sec:5} constructs fully finite-dimensional
approximations with explicit total-degree and Dirichlet-sum truncation
errors.

\section{Norm of weighted Hankel operators} \label{Sec:2}
The theory of Hankel operators is well studied, and much literature is available on this topic; see \cite{vv}. Here we consider a weighted variant of the Hankel operator.

For $0 \leq t \leq \tfrac{1}{2}$, define an operator $H_t : \ell^2(\mathbb{N}_0) \to \ell^2(\mathbb{N}_0)$ by
\[
(H_t x)_j = \sum_{k=0}^{\infty} \binom{j+k}{j} t^{j+k} x_k, \quad j \geq 0.
\]
The matrix entries of $H_t$ depend only on $j+k$, up to the binomial weight, and thus $H_t$ can be viewed as a weighted Hankel operator.

Our next objective is to investigate the boundedness of $H_t$. To this end, we first recall some useful results concerning weighted composition operators on the Hardy space $H^2$ of the open unit disk $\mathbb{D}$, consisting  of power series with square summable coefficients, which will play a crucial role in establishing the boundedness of $H_t$.

Given analytic functions $\phi : \mathbb{D} \to \mathbb{D}$ and $\psi : \mathbb{D} \to \mathbb{C}$, the weighted composition operator $W_{\psi,\phi}$ is defined by
\[
(W_{\psi,\phi}f)(z) = M_{\psi}(C_{\phi}f)(z)=\psi(z) f(\phi(z)), \qquad f \in H^2.
\]
If $\phi$ is the identity map, the weighted composition operator becomes a multiplication operator $M_{\psi}.$
In the classical setting, it is well known that the composition operator $C_\phi$ is bounded on $H^2$ if  $\phi$ is an analytic self-map of $\mathbb{D}$ (see Section 1.3 in \cite{shap}). Moreover, the multiplication operator $M_\psi$ is bounded on $H^2$ for every bounded analytic function $\psi$ on $\mathbb{D}$ (see Page~11 in \cite{shap}).

 The boundedness and norm of weighted composition operators are not well understood. Although various sufficient conditions and partial characterizations for boundedness are known, a simple or explicit necessary and sufficient condition  analogous to the case of composition operators is not known in general. In particular, no explicit formula for the operator norm of $W_{\psi,\phi}$ on $H^2$ is available in full generality, and obtaining sharp norm estimates remains a challenging problem. In  Corollary \ref{cor:NWCO},  we obtain explicit formulas for the norm of $W_{\psi,\phi}$ for certain classes of symbols $\phi$ and weights $\psi$.
 One may refer to \cite{shap,cow} for basic information about the composition operators on the Hardy space over $\mathbb{D}$.

\begin{proposition}\label{prop: hbd}
    For $0\leq t\leq \frac{1}{2},$ the operator $H_t$ is positive and bounded  on $\ell^2(\mathbb{N}_0),$ and hence self-adjoint.
\end{proposition}
\begin{proof}
  For each  $x=(x_k)\in \ell^2(\mathbb{N}_0)$, we identify with
$f(z)=\sum_{k=0}^{\infty} x_k z^k \in H^2(\mathbb{D}).$
Then, under this identification, $H_t$ can be viewed as an operator on $H^2(\mathbb{D})$ as follows:
\[
({H}_t f)(z)
=\sum_{j=0}^{\infty} ({H}_t x)_j z^j
= \sum_{j=0}^{\infty} \sum_{k=0}^{\infty} \binom{j+k}{j} t^{j+k} x_k z^j.
\]
We first justify the interchange of summation. For $|z|<1$,
\[
\sum_{j,k\ge0} \left|\binom{j+k}{j} t^{j+k} x_k z^j\right|
= \sum_{k=0}^{\infty} |x_k| t^k \sum_{j=0}^{\infty} \binom{j+k}{j} (t|z|)^j.
\]
Using
\[
\sum_{j=0}^{\infty} \binom{j+k}{j} r^j = \frac{1}{(1-r)^{k+1}}, \quad |r|<1,
\]
we obtain
\[
\sum_{j,k} \left|\binom{j+k}{j} t^{j+k} x_k z^j\right|
= \frac{1}{1-t|z|} \sum_{k=0}^{\infty} |x_k|
\left(\frac{t}{1-t|z|}\right)^k.
\]
Since $t\leq \frac{1}{2}$, one has $\frac{t}{1-t|z|}<1$. As $x\in \ell^2$, it follows from the Cauchy--Schwarz inequality that $\sum_{j,k\ge0} \left|\binom{j+k}{j} t^{j+k} x_k z^j\right|< \infty$. Hence the series is absolutely convergent, and we may interchange the sums:
\[
({H}_t f)(z)
= \sum_{k=0}^{\infty} x_k  t^k
\sum_{j=0}^{\infty} \binom{j+k}{j} (tz)^j.
\]
Using the same identity with $r=tz$, we obtain
\[
\sum_{j=0}^{\infty} \binom{j+k}{j} (tz)^j
= \frac{1}{(1-tz)^{k+1}},
\]
hence
\[
({H}_t f)(z)
= \sum_{k=0}^{\infty} x_k t^k \frac{1}{(1-tz)^{k+1}}
= \frac{1}{1-tz} \sum_{k=0}^{\infty} x_k \left(\frac{t}{1-tz}\right)^k.
\]
Therefore
\[
({H}_t f)(z)
= \frac{1}{1-tz} f\!\left(\frac{t}{1-tz}\right).
\]
Define
\[
{\phi}_t(z)=\frac{t}{1-tz}, \qquad
{w}_t(z)=\frac{1}{1-tz}.
\]
For $|z|<1$,
\[
|{\phi}_t(z)|=\frac{t}{|1-tz|}
\le \frac{t}{1-t|z|}<\frac{t}{1-t}\le 1,
\]
so ${\phi}_t(\mathbb{D})\subset \mathbb{D}$ and also analytic. Moreover,
\[
|{w}_t(z)|\le \frac{1}{1-t|z|}<\frac{1}{1-t}\le 2,
\]
hence ${w}_t$ is a bounded and also analytic function.\\
Since
\[
({H}_t f)(z)={w}_t(z) f({\phi}_t(z)),
\]
where ${\phi}_t$ is an analytic self-map of $\mathbb{D}$ and ${w}_t$ is a bounded analytic function on $\mathbb{D}$, it follows that ${H}_t = M_{{w}_t} C_{{\phi}_t}$ is bounded on $H^2(\mathbb{D})$.\\
For finitely supported $x$, Vandermonde's identity gives
\[
\binom{j+k}{j}
=
\sum_{\ell\geq0}
\binom{j}{\ell}\binom{k}{\ell}.
\]
Consequently,
\[
\begin{aligned}
\langle H_tx,x\rangle
&=
\sum_{j,k\geq0}
\binom{j+k}{j}t^{j+k}x_k\overline{x_j}\\
&=
\sum_{\ell\geq0}
\left|
\sum_{j\geq\ell}
\binom{j}{\ell}t^jx_j
\right|^2
\geq0.
\end{aligned}
\]
Since finitely supported sequences are dense in
$\ell^2(\mathbb N_0)$ and $H_t$ is bounded, it follows that
\[
\langle H_tx,x\rangle\geq0,
\qquad x\in\ell^2(\mathbb N_0).
\]
Thus, $H_t$ is positive.
\end{proof}
Since $H_t$ is symmetric, we will use the following symmetric form of
Schur's test to estimate its norm; see \cite[p.~24]{hal}.

\begin{lemma}(\textbf{Schur's test})
  \label{lem:Schur}
Let \(A=(a_{jk})_{j,k\geq0}\) be a symmetric matrix with nonnegative
entries.  Suppose that there exist positive numbers \(p_k\) and a
constant \(M>0\) such that
\[
\sum_{k=0}^{\infty}a_{jk}p_k
\leq
Mp_j
\qquad(j\geq0).
\]
Then \(A\) defines a bounded self-adjoint operator on
\(\ell^2(\mathbb N_0)\) and
\[
\|A\|\leq M.
\]
\end{lemma}

\begin{theorem} \label{Thm:HO}
For $0 \le t \leq \tfrac{1}{2}$, we have

\[
\|{H}_t\| = \frac{2}{1+\sqrt{1-4t^2}}.
\]
\end{theorem}

\begin{proof}
We split the proof into two cases: for $0\leq t<1/2$ and $t=\frac{1}{2}$.\\
\textbf{Case} \Romannum{1}:  For $0\leq t<1/2$.

If $t=0,$ then \(H_0\) has only one nonzero matrix entry, that is, $(H_0)_{0,0}=1.$ Thus
\[
\|H_0\|=1.
\]
Assume now that $0<t<\frac{1}{2}.$
The boundedness of $H_t$ follows from Proposition \ref{prop: hbd}. Now we use Schur's test to compute the norm.
Consider $x_k=a^k$ with $0< a<1$. Then
\[
f(z)=\frac{1}{1-az}.
\]
Substituting into the operator formula gives
\[
(H_t f)(z)
= \frac{1}{1-tz}\cdot \frac{1}{1-a\frac{t}{1-tz}}
= \frac{1}{(1-at)}\frac{1}{1-\frac{t}{1-at}z}.
\]
Thus $f$ is an eigenvector if
\[
\frac{t}{1-at}=a,
\]
which is equivalent to $t a^2 - a + t = 0.$
After solving this quadratic equation, we obtain
\[
a=\frac{1-\sqrt{1-4t^2}}{2t},
\]
where the root is chosen so that $0< a<1$. The corresponding eigenvalue is
\[
\lambda=\frac{1}{1-at}
= \frac{2}{1+\sqrt{1-4t^2}}.
\]
Hence $\|H_t\|\ge \lambda$.

For the reverse inequality, set $p_k=a^k$. Then
\[
\sum_{k=0}^{\infty} \binom{j+k}{j} t^{j+k} p_k
= t^j \sum_{k=0}^{\infty} \binom{j+k}{j} (at)^k
= t^j \frac{1}{(1-at)^{j+1}}
= \frac{1}{1-at} \left(\frac{t}{1-at}\right)^j.
\]
Using the relation $a=\frac{t}{1-at}$, this becomes
\[
= \frac{1}{1-at} a^j = \lambda p_j.
\]
By Lemma \ref{lem:Schur}, $\|H_t\|\le \lambda$. Therefore
\[
\|H_t\| = \frac{2}{1+\sqrt{1-4t^2}}, \quad 0 \le t < \tfrac12.
\]
\textbf{Case} \Romannum{2}: For $t=\frac{1}{2}.$
Recall \[
(H_t)_{j,k}=\binom{j+k}{j}t^{j+k}.
\]
Apply Schur's test with constant weight $1$,
\[
\sum_{k=0}^{\infty}\binom{j+k}{j}2^{-j-k}=2.
\]
Therefore  $\|H_t\|\leq2.$

For the reverse inequality, fix \(0<t<1/2\). Since \(H_{1/2}-H_t\) has nonnegative entries and the vector $(a^j)_{j\geq0}$ has nonnegative entries, we have
\[
\left\|H_{1/2}\right\|
\geq
\frac{\left\langle H_{1/2}(a^j)_{j\geq 0},(a^j)_{j\geq 0}\right\rangle}
{\left\|(a^j)_{j\geq 0}\right\|^2}
\geq
\frac{\left\langle H_t(a^j)_{j\geq 0},(a^j)_{j\geq 0}\right\rangle}
{\left\|(a^j)_{j\geq 0}\right\|^2}.
\]
From the first case, the last expression is
\[
\frac{2}{1+\sqrt{1-4t^2}}.
\]
Letting $t\to \frac{1}{2}$, gives $\|H_{1/2}\|\geq2.$ Thus $\|H_{1/2}\|=2.$
\end{proof}

\begin{corollary}[Signed Hankel operator]\label{cor:SHO}
Let $0 \le t \le \tfrac{1}{2}$ and define $\widetilde{H}_t : \ell^2(\mathbb{N}_0) \to \ell^2(\mathbb{N}_0)$ by
\[
(\widetilde{H}_t x)_j = \sum_{k=0}^{\infty} (-1)^{j+k} \binom{j+k}{j} t^{j+k} x_k, \quad j \ge 0.
\]
Then $\widetilde{H}_t$ is a bounded operator with
\[
\|\widetilde{H}_t\| = \frac{2}{1+\sqrt{1-4t^2}}.
\]
\end{corollary}
\begin{proof}
Define $U : \ell^2(\mathbb{N}_0) \to \ell^2(\mathbb{N}_0)$ by $(Ux)_k = (-1)^k x_k$. Then $U$ is unitary, since
\[
\|Ux\|^2 = \sum_{k\ge 0} |(-1)^k x_k|^2 = \sum_{k\ge 0} |x_k|^2 = \|x\|^2,
\]
and $U^{-1} = U=U^*$.
A direct computation shows that for $x \in \ell^2(\mathbb{N}_0)$,
\[
(U H_t U x)_j
= (-1)^j \sum_{k\ge 0} \binom{j+k}{j} t^{j+k} (-1)^k x_k
= (\widetilde{H}_t x)_j.
\]
Thus $\widetilde{H}_t = U H_t U$, and hence $H_t$ and $\widetilde{H}_t$ are unitarily equivalent. In particular, they have the same norm.
\end{proof}

 \begin{corollary} \label{cor:NWCO}
Let $0 \le t \leq \tfrac{1}{2}$, and define the analytic functions on the unit disk $\mathbb{D}$ by
\[
\varphi(z) = \frac{t}{1-tz}, \qquad \psi(z) = \frac{1}{1-tz}.
\]
Then the weighted composition operator $W_{\psi,\varphi}$ defined on $H^2(\mathbb{D})$ by
\[
(W_{\psi,\varphi} f)(z) = \psi(z)\, f(\varphi(z)), \qquad f \in H^2(\mathbb{D}),
\]
is bounded on $H^2(\mathbb{D})$, and its operator norm is given by
\[
\|W_{\psi,\varphi}\|
=
\frac{2}{1+\sqrt{1-4t^2}}.
\]
\end{corollary}

\begin{proof}
  The proof follows along the lines of the proof of Proposition \ref{prop: hbd} by identifying
  $W_{\psi,\varphi}$ as a Hankel operator.
\end{proof}

\section{Finite-prime symbols}
\label{sec:3}

Let $p_1,\ldots,p_d$ be distinct primes, and let
\[
\varphi(s)
=
\sigma+\sum_{j=1}^{d}r_jp_j^{-s}.
\]
By the Gordon$-$Hedenmalm theorem (see Theorem B in \cite{GH}), \(\varphi\) induces a bounded composition operator on \(\mathcal{H}^{2}\) if and only if $\operatorname{Re}\sigma>1/2$ and $\operatorname{Re}\sigma-1/2\geq \sum_{j=1}^{d}|r_j|$.
Note that  for $\operatorname{Re}s>0$, we have $|p_j^{-s}|<1,$ for $1\leq j\leq d.$ Hence
\[
\operatorname{Re}\varphi(s)
\geq
\operatorname{Re}\sigma-\sum_{j=1}^{d}|r_j||p_j^{-s}|
>
\operatorname{Re}\sigma-\sum_{j=1}^{d}|r_j|
\geq
\frac{1}{2},
\]
unless all \(r_j=0\), in which case
$\operatorname{Re}\varphi(s)=\operatorname{Re}\sigma>\frac{1}{2}.$
Thus $  \varphi(\mathbb{C}_{0})\subset \mathbb{C}_{1/2}.$

Recall from the Section~2 in \cite{BP1} that $\|C_{\varphi}\|=\|C_{\psi}\|$, where $\psi = \operatorname{Re}\sigma+ \sum_{j=1}^d|r_j|p_j^{-s}.$ Hence we may assume that
\begin{equation}\label{eq:finite-prime symbol}
    \varphi(s)=\sigma+\sum_{j=1}^{d}r_jp_j^{-s},
\text{ where } \sigma\in(1/2, \infty), \quad r_j\ge0, \quad \sigma -\sum_{j=1}^{d}r_j\geq\frac{1}{2}
\end{equation}

In this section, we analyze the
norm of composition operators induced by finite-prime symbols. Our approach is to replace the composition operator acting on $\mathcal{H}^2$ by an equivalent coefficient operator on an $\ell^{2}$-space. This realization transforms the norm problem into the study of an explicit positive operator.

For a multi-index $\alpha=(\alpha_1,\ldots,\alpha_d)\in\mathbb{N}_{0}^{d},$
we use the standard notation
\[
|\alpha|
=
\alpha_1+\cdots+\alpha_d,
\qquad
\alpha!
=
\alpha_1!\cdots\alpha_d!.
\]
If $\boldsymbol{r}=(r_1,\ldots,r_d)\in[0,\infty)^d$,
then
$\boldsymbol{r}^{\alpha}
=
r_1^{\alpha_1}\cdots r_d^{\alpha_d}$ and if $p_1, p_2, \cdots,p_d$ are distinct primes, then $p^{\alpha} = p_1^{\alpha_1}\cdots p_d^{\alpha_d}.$ Throughout, we use the convention $0^0=1.$

Our first step is a coefficient realization for finite-prime symbols $\varphi.$ Let
\[
f(s)=\sum_{n=1}^{N}a_n n^{-s}.
\] be a Dirichlet polynomial.
Then
\[
\begin{aligned}
(C_{\varphi}f)(s)
&=
f\left(\sigma+\sum_{j=1}^{d}r_jp_j^{-s}\right)  \\
&=
\sum_{n=1}^{N}
a_n n^{-\sigma}
\exp\left(
-(\log n)\sum_{j=1}^{d}r_jp_j^{-s}
\right).
\end{aligned}
\]
For each fixed \(n\), we expand
\[
\begin{aligned}
\exp\left(
-(\log n)\sum_{j=1}^{d}r_jp_j^{-s}
\right)
&=
\prod_{j=1}^{d}
\exp\left(-r_j(\log n)p_j^{-s}\right) \\
&=
\prod_{j=1}^{d}
\sum_{\alpha_j=0}^{\infty}
\frac{(-r_j\log n)^{\alpha_j}}{\alpha_j!}p_j^{-\alpha_js} \\
&=
\sum_{\alpha\in\mathbb{N}_{0}^{d}}
\frac{(-\log n)^{|\alpha|}\boldsymbol{r}^{\alpha}}{\alpha!}
(p^{\alpha})^{-s}.
\end{aligned}
\]
For each fixed $n$, the above series converges absolutely, since
\[
\begin{aligned}
\sum_{\alpha\in\mathbb N_0^d}
\left|
\frac{(-\log n)^{|\alpha|}r^\alpha}{\alpha!}
(p^\alpha)^{-s}
\right|
&=
\prod_{j=1}^d
\sum_{\alpha_j=0}^{\infty}
\frac{\big(r_jp_j^{-\operatorname{Re}s}\log n\big)^{\alpha_j}}
{\alpha_j!}\\
&=
\exp\left((\log n)\sum_{j=1}^d
r_jp_j^{-\operatorname{Re}s}\right)
<\infty.
\end{aligned}
\]
Since the sum over $n$ is finite, we may interchange the sums over
$n$ and $\alpha$. Thus,
\begin{equation}\label{eq: finite coefficient}
C_{\varphi}f
=
\sum_{\alpha\in\mathbb{N}_{0}^{d}}
\left(
\sum_{n=1}^{N}
a_n n^{-\sigma}
\frac{(-\log n)^{|\alpha|}\boldsymbol{r}^{\alpha}}{\alpha!}
\right)
(p^{\alpha})^{-s}.
\end{equation}
Define the coefficient operator
\[
T_{\sigma,\boldsymbol{r}}:\ell^2(\mathbb{N})\longrightarrow\ell^2(\mathbb{N}_{0}^{d})
\]
 by \[ a=( a_n) \mapsto \bigl((T_{\sigma,\boldsymbol r}a)_\alpha\bigr)_{
\alpha\in\mathbb N_0^d},\]  where
\[
(T_{\sigma,\boldsymbol{r}}a)_{\alpha}
=
\sum_{n=1}^{\infty}
a_n n^{-\sigma}
\frac{(-\log n)^{|\alpha|}\boldsymbol{r}^{\alpha}}{\alpha!}.
\]

\begin{proposition}
\label{prop:fp-coefficient-realization}
Let $\varphi,\sigma,\boldsymbol{r}$ be as in \eqref{eq:finite-prime symbol}. Then the operator \(T_{\sigma,\boldsymbol{r}}\) is  bounded on $\ell^2(\mathbb{N})$ with
\[
\left\|C_{\varphi}\right\|
=
\left\|T_{\sigma,\boldsymbol{r}}\right\|.
\]
\end{proposition}

\begin{proof}
For distinct primes $p_1,\ldots,p_d$, define the finite-prime subspace
\[
\mathcal{H}^{2}(p_1,\ldots,p_d)
=
\overline{\operatorname{span}}
\left\{
(p^{\alpha})^{-s}:\alpha\in\mathbb{N}_{0}^{d}
\right\}
\subset \mathcal{H}^{2}.
\]
Define
\[
U_d:\ell^2(\mathbb{N}_{0}^{d})\to \mathcal{H}^{2}(p_1,\ldots,p_d)
\]
by
\[
U_d b
=
\sum_{\alpha\in\mathbb{N}_{0}^{d}}b_{\alpha}(p^{\alpha})^{-s}.
\]
Since $\alpha\mapsto p^{\alpha}$ is injective, \(U_d\) is unitary.
Then equation  \eqref{eq: finite coefficient} gives
\[
C_{\varphi}f
=
U_dT_{\sigma,\boldsymbol{r}}a, \qquad \text{ for } a\in C_{00}.
\]
Thus \[
\left\|C_{\varphi}f\right\|_{\mathcal{H}^{2}}
=
\left\|T_{\sigma,\boldsymbol{r}}a\right\|_{\ell^2(\mathbb{N}_{0}^{d})}.
\]
Because \(C_{\varphi}\) is bounded on \(\mathcal{H}^{2}\), we get, for every finitely supported \(a\),
\[
\left\|T_{\sigma,\boldsymbol{r}}a\right\|_{\ell^2(\mathbb{N}_{0}^{d})}
=
\left\|C_{\varphi}f\right\|_{\mathcal{H}^{2}}
\leq
\left\|C_{\varphi}\right\|\left\|f\right\|_{\mathcal{H}^{2}}
=
\left\|C_{\varphi}\right\|\left\|a\right\|_{\ell^2(\mathbb{N})}.
\]
Hence \(T_{\sigma,\boldsymbol{r}}\) extends uniquely to a bounded operator (see Theorem 1.9.1 in \cite{RM})
\[
\widetilde{T}_{\sigma,\boldsymbol{r}}:\ell^2(\mathbb{N})\to\ell^2(\mathbb{N}_{0}^{d}),
\]
whose restriction to $c_{00}$ coincides with $T_{\sigma,\boldsymbol{r}}$.
Explicitly, if
\[
a^{(m)}\in c_{00},
\qquad
a^{(m)}\rightarrow a
\quad\text{in }\ell^2(\mathbb N),
\]
then
\[
\widetilde T_{\sigma,\boldsymbol{r}}a
=
\lim_{m\rightarrow\infty}
T_{\sigma,\boldsymbol{r}}a^{(m)}.
\]
By uniqueness, we identify $\widetilde T_{\sigma,\boldsymbol{r}}$ with
$T_{\sigma,\boldsymbol{r}}$.
Thus
$\left\|T_{\sigma,\boldsymbol{r}}\right\|
\leq
\left\|C_{\varphi}\right\|.$

Conversely, for every Dirichlet polynomial \(f\) with coefficient vector \(a\), the identity
\[
C_{\varphi}f
=
U_dT_{\sigma,\boldsymbol{r}}a
\]
gives
\[
\left\|C_{\varphi}f\right\|_{\mathcal{H}^{2}}
=
\left\|T_{\sigma,\boldsymbol{r}}a\right\|_{\ell^2(\mathbb{N}_{0}^{d})}
\leq
\left\|T_{\sigma,\boldsymbol{r}}\right\|
\left\|a\right\|_{\ell^2(\mathbb{N})}
=
\left\|T_{\sigma,\boldsymbol{r}}\right\|
\left\|f\right\|_{\mathcal{H}^{2}}.
\]
Since Dirichlet polynomials are dense in \(\mathcal{H}^{2}\), and  \(C_{\varphi}\) is bounded, it follows that
\[
\left\|C_{\varphi}\right\|
\leq
\left\|T_{\sigma,\boldsymbol{r}}\right\|.
\]
Therefore
\[
\left\|C_{\varphi}\right\|
=
\left\|T_{\sigma,\boldsymbol{r}}\right\|.
\]
It remains to show that the extended operator is still given by the same coefficient formula. We have
\[
(T_{\sigma,\boldsymbol{r}}a^{(m)})_\alpha
=
\sum_{n=1}^{\infty}
a_n^{(m)}
n^{-\sigma}
\frac{(-\log n)^{|\alpha|}}{\alpha!}
\boldsymbol{r}^\alpha.
\]

Therefore,
\[
\begin{aligned}
&
\left|
(T_{\sigma,\boldsymbol{r}}a^{(m)})_\alpha
-
\sum_{n=1}^{\infty}
a_n
n^{-\sigma}
\frac{(-\log n)^{|\alpha|}}{\alpha!}
\boldsymbol{r}^\alpha
\right|
\\
&\le
\sum_{n=1}^{\infty}
|a_n^{(m)}-a_n|
n^{-\sigma}
\frac{(\log n)^{|\alpha|}}{\alpha!}
|\boldsymbol{r}^\alpha|.
\end{aligned}
\]

Applying the Cauchy--Schwarz inequality gives
\[
\begin{aligned}
&
\left|
(T_{\sigma,\boldsymbol{r}}a^{(m)})_\alpha
-
\sum_{n=1}^{\infty}
a_n
n^{-\sigma}
\frac{(-\log n)^{|\alpha|}}{\alpha!}
\boldsymbol{r}^\alpha
\right|
\\
&\le
\|a^{(m)}-a\|_{\ell^2}
\left(
\sum_{n=1}^{\infty}
n^{-2\sigma}
(\log n)^{2|\alpha|}
\right)^{1/2}
\frac{|\boldsymbol{r}^\alpha|}{\alpha!}.
\end{aligned}
\]

Since $\sigma>\frac12$, the series
\[
\sum_{n=1}^{\infty}
n^{-2\sigma}
(\log n)^{2|\alpha|}
\]
converges, while
\[
a^{(m)}\rightarrow a
\quad\text{in }\ell^2(\mathbb N).
\]
Hence the right-hand side tends to zero as $m\rightarrow\infty$, and therefore
\[
(T_{\sigma,\boldsymbol{r}}a)_\alpha
=
\sum_{n=1}^{\infty}
a_n
n^{-\sigma}
\frac{(-\log n)^{|\alpha|}}{\alpha!}
\boldsymbol{r}^\alpha,
\]
which proves that the bounded extension is represented by the same
coefficient formula.
\end{proof}

Proposition \ref{prop:fp-coefficient-realization} identifies the composition operator with its coefficient realization. Consequently, the norm problem can be formulated entirely in terms of the coefficient operator $T_{\sigma,r}.$
Since $\|T_{\sigma,r}\|^2=\|T_{\sigma,r}T_{\sigma,r}^*\|$, (see Section $12.9$ in \cite{Ru91})
we may study the norm problem through the positive operator
$A_{\sigma,\boldsymbol r}
:=
T_{\sigma,\boldsymbol r}T_{\sigma,\boldsymbol r}^*$. We first derive an explicit expression for the matrix of $A_{\sigma,\boldsymbol r}$, which will play a central role in the subsequent  analysis. For a given operator $A$ on $\ell^2(\mathbb{N}_{0}^{d}),$ $(A)_{\alpha,\beta}$ denotes the $(\alpha,\beta)$-th entry of the matrix representation of $A.$

\begin{proposition}
\label{prop:fp-output-kernel}
Let $\sigma, \boldsymbol{r}$ be as in \eqref{eq:finite-prime symbol}. Then
for $\alpha,\beta\in\mathbb{N}_{0}^{d}$,
we have
\[
(T_{\sigma,\boldsymbol{r}}T_{\sigma,\boldsymbol{r}}^{*})_{\alpha,\beta}
=
\frac{\boldsymbol{r}^{\alpha+\beta}}{\alpha!\beta!}
\zeta^{(|\alpha|+|\beta|)}(2\sigma).
\]
\end{proposition}

\begin{proof}
Let $(\varepsilon_n)_{n\geq 1}$ be the standard orthonormal basis of $\ell^2(\mathbb{N}),$ and let $(e_{\alpha})_{\alpha\in\mathbb{N}_{0}^{d}}$ be the standard orthonormal basis of $\ell^2(\mathbb{N}_{0}^{d}).$

The matrix entries of \(T_{\sigma,\boldsymbol{r}}\) are
\[
\left\langle T_{\sigma,\boldsymbol{r}}\varepsilon_n,e_{\alpha}\right\rangle
=
n^{-\sigma}
\frac{(-\log n)^{|\alpha|}\boldsymbol{r}^{\alpha}}{\alpha!}.
\]
Then the adjoint of $T_{\sigma,r} $ is given by
\[
T_{\sigma,\boldsymbol{r}}^{*}e_{\alpha}
=
\left(
n^{-\sigma}
\frac{(-\log n)^{|\alpha|}\boldsymbol{r}^{\alpha}}{\alpha!}
\right)_{n\geq 1}.
\]
Therefore
\[
\begin{aligned}
(T_{\sigma,\boldsymbol{r}}T_{\sigma,\boldsymbol{r}}^{*})_{\alpha,\beta}
&=
\left\langle
T_{\sigma,\boldsymbol{r}}T_{\sigma,\boldsymbol{r}}^{*}e_{\beta},
e_{\alpha}
\right\rangle                                                      \\
&=
\left\langle
T_{\sigma,\boldsymbol{r}}^{*}e_{\beta},
T_{\sigma,\boldsymbol{r}}^{*}e_{\alpha}
\right\rangle_{\ell^2(\mathbb{N})}                                  \\
&=
\sum_{n=1}^{\infty}
n^{-2\sigma}
\frac{(-\log n)^{|\alpha|+|\beta|}
\boldsymbol{r}^{\alpha+\beta}}
{\alpha!\beta!}.
\end{aligned}
\]
Recall that for $\operatorname{Re}s>1$ and $m\in\mathbb{N}_{0},$ we have
\[
\zeta^{(m)}(s)
=
\sum_{n=1}^{\infty}
(-\log n)^m n^{-s}.
\]
Applying this with
$m=|\alpha|+|\beta|,$ and $s=2\sigma,$ gives
\[
\sum_{n=1}^{\infty}
n^{-2\sigma}
(-\log n)^{|\alpha|+|\beta|}
=
\zeta^{(|\alpha|+|\beta|)}(2\sigma).
\]
Hence
\[
(T_{\sigma,\boldsymbol{r}}T_{\sigma,\boldsymbol{r}}^{*})_{\alpha,\beta}
=
\frac{\boldsymbol{r}^{\alpha+\beta}}{\alpha!\beta!}
\zeta^{(|\alpha|+|\beta|)}(2\sigma).
\]
\end{proof}
The above coefficient representation reduces the norm problem to the positive
operator $A_{\sigma,\boldsymbol r}$. We now examine
this operator in the singular regime in which the constant term
approaches the boundary $\operatorname{Re}s=1/2$ and the prime
coefficients vanish on the same scale.

From \eqref{eq:finite-prime symbol}, we have $\sigma>1/2$ and $\sigma-1/2\geq(r_1+r_2+\cdots+r_d).$  Set
\[\delta:=\sigma-1/2>0, \qquad \rho_j:=\frac{r_j}{\delta}\]
Then $0\leq\rho_1 +\cdots+\rho_d\leq1$, and the  symbol $\varphi$ can be written as
\begin{equation}\label{eq:scaling symbol}
\varphi_{\delta,\boldsymbol{\rho}}(s)
=
\frac{1}{2}+\delta+\delta\sum_{j=1}^{d}\rho_jp_j^{-s},
\end{equation}
where $\boldsymbol{\rho}=(\rho_1, \rho_2,\ldots,\rho_d).$

This scaling isolates the boundary parameter $\delta$, allowing us to study the asymptotic behavior as $\delta\downarrow0$ while keeping the normalized parameters $\boldsymbol{\rho}$ fixed.
Recall that $B_d=\left\{\boldsymbol{\rho}=(\rho_1, \ldots, \rho_d)\in[0,1]^d : \rho_1+\cdots+\rho_d\leq1 \right\}.$ It is clear that, for any $\delta>0$ and $\boldsymbol{\rho}\in B_d,$ the composition operator induced by the symbol $\varphi_{\delta,\boldsymbol{\rho}}$ is always
bounded on $\mathcal{H}^2$.
For $\boldsymbol{\rho}\in B_d,$
we use the notation
$\boldsymbol{\rho}^{\alpha}
= \rho_1^{\alpha_1}\cdots \rho_d^{\alpha_d}$, and $R_{\boldsymbol{\rho}}:=\rho_1+\cdots+\rho_d.$

Since $\delta=\sigma-1/2$ and $\boldsymbol{r}=\delta \boldsymbol{\rho},$ the operator $T_{\sigma,\boldsymbol{r}}$ introduced above, under this scaling, will
be denoted by
\[
T_{\delta,\boldsymbol{\rho}}
=
T_{\frac{1}{2}+\delta,\delta\boldsymbol{\rho}},
\qquad
A_{\delta,\boldsymbol{\rho}}
=
T_{\delta,\boldsymbol{\rho}}T_{\delta,\boldsymbol{\rho}}^{*}.
\]
Thus
\[
(T_{\delta,\boldsymbol{\rho}}a)_{\alpha}
=
\sum_{n=1}^{\infty}
a_n n^{-1/2-\delta}
\frac{(-\delta\log n)^{|\alpha|}\boldsymbol{\rho}^{\alpha}}{\alpha!}.
\]
By Proposition \ref{prop:fp-output-kernel}, we get the following result
\begin{equation}\label{eq:positive operator A}
(A_{\delta,\boldsymbol{\rho}})_{\alpha,\beta}
=
\frac{
\delta^{|\alpha|+|\beta|}
\boldsymbol{\rho}^{\alpha+\beta}}
{\alpha!\beta!}
\zeta^{(|\alpha|+|\beta|)}(1+2\delta).
\end{equation}
For $\boldsymbol{\rho}=(\rho_1,\ldots,\rho_d)\in[0, 1]^d,$ define  a multivariate weighted Hankel matrix $\mathcal{H}_{\boldsymbol{\rho}}
= \left((\mathcal{H}_{\boldsymbol{\rho}})_{\alpha,\beta}
\right)_{\alpha,\beta\in\mathbb{N}_{0}^{d}},$ where
\begin{equation}\label{eq:multivariate HO}
(\mathcal{H}_{\boldsymbol{\rho}})_{\alpha,\beta}
=
(-1)^{|\alpha|+|\beta|}
\frac{(|\alpha|+|\beta|)!}{\alpha!\beta!}
\prod_{j=1}^{d}
\left(\frac{\rho_j}{2}\right)^{\alpha_j+\beta_j}
\end{equation}
which will serve as the principal tool in the proof of the main theorem.

When $R_{\boldsymbol{\rho}}=0$, the operator
$\mathcal H_{\boldsymbol0}$ is rank one with the norm $1$. Hence, unless otherwise stated, we assume $R_{\boldsymbol{\rho}}>0.$
The next lemma proves the boundedness of
$\mathcal H_{\boldsymbol{\rho}}$.

\begin{lemma}
\label{lem:fp-boundedness-H-rho}
For a given $\boldsymbol{\rho}\in B_d,$  the matrix $\mathcal{H}_{\boldsymbol{\rho}}$ defines a bounded  operator on $\ell^2(\mathbb{N}_{0}^{d}).$
Moreover,
\[
\left\|\mathcal{H}_{\boldsymbol{\rho}}\right\|
\leq
\frac{2}{1+\sqrt{1-R_{\boldsymbol{\rho}}^2}}.
\]
\end{lemma}
\begin{proof}
Assume that $R_{\boldsymbol{\rho}}>0,$ and consider
 $\theta_j = \frac{\rho_j}{R_{\boldsymbol{\rho}}},$  for $ 1\leq j\leq d.$

For $|\alpha|=N,$ define
\[
w_{\alpha}
=
\frac{N!}{\alpha!}
\theta_1^{\alpha_1}\cdots \theta_d^{\alpha_d}.
\]
Then, by using multinomial theorem
\[
w_{\alpha}\geq 0,
\qquad
\sum_{|\alpha|=N}w_{\alpha}=1.
\]
Define $V:\ell^2(\mathbb{N}_{0}^{d})\to \ell^2(\mathbb{N}_{0})$
by
\[
(Vx)_N
=
\sum_{|\alpha|=N}w_{\alpha}x_{\alpha}.
\]
For each \(N\), the Cauchy--Schwarz inequality gives
\[
|(Vx)_N|^2
\leq
\left(\sum_{|\alpha|=N}w_{\alpha}^2\right)
\left(\sum_{|\alpha|=N}|x_{\alpha}|^2\right).
\]
Since
\[
\sum_{|\alpha|=N}w_{\alpha}^2
\leq
\left(\sum_{|\alpha|=N}w_{\alpha}\right)^2
=
1,
\]
we obtain
\[
|(Vx)_N|^2
\leq
\sum_{|\alpha|=N}|x_{\alpha}|^2.
\]
Summing over \(N\) gives
\[
\left\|Vx\right\|_{\ell^2(\mathbb{N}_{0})}
\leq
\left\|x\right\|_{\ell^2(\mathbb{N}_{0}^{d})}.
\]
Thus \(V\) is a contraction.

Let $\widetilde{H}_{R_{\boldsymbol{\rho}}/2}$ be the signed Hankel operator from  Corollary \ref{cor:SHO}. We claim that
\[
\mathcal{H}_{\boldsymbol{\rho}}
=
V^{*}\widetilde{H}_{R_{\boldsymbol{\rho}}/2}V.
\]
Indeed, let $|\alpha|=N, |\beta|=M$. Then the \((\alpha,\beta)\)-entry of
$V^{*}\widetilde{H}_{R_{\boldsymbol{\rho}}/2}V$ is
\[
w_{\alpha}
(-1)^{N+M}
\binom{N+M}{N}
\left(\frac{R_{\boldsymbol{\rho}}}{2}\right)^{N+M}
w_{\beta}.
\]
Substituting the definitions of \(w_{\alpha}\) and \(w_{\beta}\), this becomes
\[
(-1)^{N+M}
\frac{(N+M)!}{\alpha!\beta!}
\prod_{j=1}^{d}
\left(\frac{\rho_j}{2}\right)^{\alpha_j+\beta_j}.
\]
This is exactly $(\mathcal{H}_{\boldsymbol{\rho}})_{\alpha,\beta}.$
Therefore
$\mathcal{H}_{\boldsymbol{\rho}}
=
V^{*}\widetilde{H}_{R_{\boldsymbol{\rho}}/2}V$
is a bounded operator.
Since \(V\) is a contraction, Corollary \ref{cor:SHO} gives
\[
\left\|\mathcal{H}_{\boldsymbol{\rho}}\right\|
\leq
\left\|\widetilde{H}_{R_{\boldsymbol{\rho}}/2}\right\|
=
\frac{2}{1+\sqrt{1-R_{\boldsymbol{\rho}}^2}}.
\]
\end{proof}

The following theorem establishes the asymptotic behavior of the norm of the composition operator.

\begin{theorem}\label{main}
Fix \(\eta>0\). Then there exists a constant \(C_\eta>0\) such that,
for every \(0<\delta\leq\eta\),
\[
\sup_{\boldsymbol\rho\in B_d}
\left\|
2\delta A_{\delta,\boldsymbol\rho}
-
\mathcal H_{\boldsymbol\rho}
\right\|
\leq C_\eta\delta.
\]
Consequently,
\[
\sup_{\boldsymbol\rho\in B_d}
\left|
2\delta\|C_{\varphi_{\delta,\boldsymbol\rho}}\|^2
-
\|\mathcal H_{\boldsymbol\rho}\|
\right|
\leq C_\eta\delta.
\]
\end{theorem}

\begin{proof}

For any $\delta>0,\; \boldsymbol{\rho}\in B_d,$ Equation \eqref{eq:positive operator A} gives,
\[
(A_{\delta,\boldsymbol{\rho}})_{\alpha,\beta}
=
\frac{
\delta^m
\boldsymbol{\rho}^{\alpha+\beta}}
{\alpha!\beta!}
\zeta^{(m)}(1+2\delta), \qquad \text{ where } m=|\alpha|+|\beta|.
\]
Define
\[
g(z)
=
\zeta(z)-\frac{1}{z-1}.
\]
Then \(g\) is an entire function and, for $m\in\mathbb{N}_{0},$ we have
\[
\zeta^{(m)}(z)
=
\frac{(-1)^m m!}{(z-1)^{m+1}}
+
g^{(m)}(z).
\]
Therefore
\[
\begin{aligned}
2\delta(A_{\delta,\boldsymbol{\rho}})_{\alpha,\beta}
&=
2\delta
\frac{
\delta^m
\boldsymbol{\rho}^{\alpha+\beta}}
{\alpha!\beta!}
\left(
\frac{(-1)^m m!}{(2\delta)^{m+1}}
+
g^{(m)}(1+2\delta)
\right)                                                     \\
&=
(-1)^m
\frac{m!}{\alpha!\beta!}
\prod_{j=1}^{d}
\left(\frac{\rho_j}{2}\right)^{\alpha_j+\beta_j}
+
(E_{\delta,\boldsymbol{\rho}})_{\alpha,\beta}.
\end{aligned}
\]
The first term in the above sum is exactly $(\mathcal{H}_{\boldsymbol{\rho}})_{\alpha,\beta}$ and
\[
(E_{\delta,\boldsymbol{\rho}})_{\alpha,\beta}
=
2\delta
\frac{
\delta^m
\boldsymbol{\rho}^{\alpha+\beta}}
{\alpha!\beta!}
g^{(m)}(1+2\delta).
\]
Thus, entrywise,
\[
2\delta A_{\delta,\boldsymbol{\rho}}
-
\mathcal{H}_{\boldsymbol{\rho}}
=
E_{\delta,\boldsymbol{\rho}}.
\]
By Lemma \ref{lem:fp-boundedness-H-rho} and Proposition
\ref{prop:fp-coefficient-realization},
$\mathcal{H}_{\boldsymbol{\rho}}$ and
$A_{\delta,\boldsymbol{\rho}} = T_{\delta,\boldsymbol{\rho}} T_{\delta,\boldsymbol{\rho}}^{*}$
are  bounded.
Hence $E_{\delta,\boldsymbol{\rho}}$ defines a bounded operator, and \[
2\delta A_{\delta,\boldsymbol{\rho}}
-
\mathcal{H}_{\boldsymbol{\rho}}
=
E_{\delta,\boldsymbol{\rho}}.
\]
It remains to estimate
$E_{\delta,\boldsymbol{\rho}}$ in the operator norm.

Fix $\eta>0$.
Define the set
\[
K_{\eta} =\bigcup_{0\leq u \leq \eta}\left\{z\in\mathbb{C}:|z-(1+2u)|=3\eta
\right\}.
\]
Since \(g\) is an entire function and \(K_\eta\) is compact, the number $M_\eta := \sup_{z\in K_\eta}|g(z)|$ is finite.
For $0<\delta\leq \eta,$ applying  Cauchy's estimate  to \(g\) on the circle
$|z-(1+2\delta)|=3\eta,$
gives
\[
|g^{(m)}(1+2\delta)|
\leq
M_\eta\frac{m!}{(3\eta)^m}.
\]
Hence, for  $0<\delta\leq \eta,$
\[
|(E_{\delta,\boldsymbol{\rho}})_{\alpha,\beta}|
\leq
2\delta M_\eta
\frac{m!}{\alpha!\beta!}
\prod_{j=1}^{d}
\left(\frac{\delta\rho_j}{3\eta}\right)^{\alpha_j+\beta_j}.
\]
Next we compute the Hilbert--Schmidt norm of the operator $E_{\delta,\rho},$ given by
\[
\left\|E_{\delta,\boldsymbol{\rho}}\right\|_{\mathrm{HS}}^2
=
\sum_{\alpha\in\mathbb{N}_{0}^{d}}
\sum_{\beta\in\mathbb{N}_{0}^{d}}
|(E_{\delta,\boldsymbol{\rho}})_{\alpha,\beta}|^{2}.
\]
Define $\gamma=(\alpha_1, \cdots,\alpha_d, \beta_1, \cdots, \beta_d)\in \mathbb{N}_{0}^{2d}.$ Then $|\gamma|=|\alpha|+|\beta|=m.$

Set
\[
x_j=\frac{\delta\rho_j}{3\eta},
\text{ and }
x_{d+j}=\frac{\delta\rho_j}{3\eta},
\qquad
1\leq j\leq d.
\]
Then the Hilbert--Schmidt norm is
\[
\left\|E_{\delta,\boldsymbol{\rho}}\right\|_{\mathrm{HS}}^2\leq(2\delta M_{\eta})^2\sum_{m=0}^{\infty}\sum_{|\gamma|=m}\left(\frac{m!}{\gamma!}x^{\gamma} \right)^2.
\]
On applying multinomial theorem, for every $m\in\mathbb{N}_{0},$ we have
\[
\sum_{\substack{\gamma\in\mathbb{N}_{0}^{2d}\\|\gamma|=m}}
\left(
\frac{m!}{\gamma!}
x_1^{\gamma_1}\cdots x_{2d}^{\gamma_{2d}}
\right)^2                                                       \leq
\left(
\sum_{\substack{\gamma\in\mathbb{N}_{0}^{2d}\\|\gamma|=m}}
\frac{m!}{\gamma!}
x_1^{\gamma_1}\cdots x_{2d}^{\gamma_{2d}}
\right)^2
 =(x_1+\cdots+x_{2d})^{2m}.
\]
Since
$x_1+\cdots+x_{2d} = \frac{2\delta R_{\boldsymbol{\rho}}}{3\eta},$
we obtain
\[
\left\|E_{\delta,\boldsymbol{\rho}}\right\|_{\mathrm{HS}}^2\leq
(2\delta M_\eta)^2
\sum_{m=0}^{\infty}
\left(
\frac{2\delta R_{\boldsymbol{\rho}}}{3\eta}
\right)^{2m}
\]
For any $\boldsymbol{\rho}\in B_d,$ we have $R_{\boldsymbol{\rho}}\leq1$ and  consequently, $\frac{2\delta R_{\boldsymbol{\rho}}}{3\eta}\leq 2/3$. Therefore
\[
\left\|E_{\delta,\boldsymbol{\rho}}\right\|
\leq
\left\|E_{\delta,\boldsymbol{\rho}}\right\|_{\mathrm{HS}}
\leq
\frac{2\delta M_\eta}
{\sqrt{1-\left(2\delta R_{\boldsymbol{\rho}}/3\eta\right)^2}}\leq\frac{6M_\eta}{\sqrt{5}}\delta.
\]
Thus
\[
\left\|
2\delta A_{\delta,\boldsymbol{\rho}}
-
\mathcal{H}_{\boldsymbol{\rho}}
\right\|
\leq
C_\eta\delta, \qquad \text{where } C_\eta=\frac{6M_\eta}{\sqrt{5}}.
\]
Since $\boldsymbol{\rho}\in B_d$ was arbitrary and
\[
\left|
2\delta\|C_{\varphi_{\delta,\boldsymbol\rho}}\|^2
-
\|\mathcal H_{\boldsymbol\rho}\|
\right|
\le
\left\|
2\delta A_{\delta,\boldsymbol\rho}
-
\mathcal H_{\boldsymbol\rho}
\right\|,
\]
the claimed uniform estimate follows.

\end{proof}
As an immediate consequence of Theorem \ref{main}, the following result holds.

\begin{corollary}\label{cor}
For  any $\boldsymbol{\rho}\in B_d$, we have
\[
2\delta\,
\bigl\|C_{\varphi_{\delta,\boldsymbol\rho}}\bigr\|^2
\longrightarrow
\|\mathcal{H}_{\boldsymbol{\rho}}\|
\qquad \text{ as }\delta\downarrow0.
\]

\end{corollary}

\medskip
\noindent\textbf{The one-prime case.}
We now specialize the critical limit to a one-prime direction; so we consider $\phi(s)=\sigma+r2^{-s},$ and, as a consequence of Corollary \ref{cor}, obtain the following asymptotic formula.

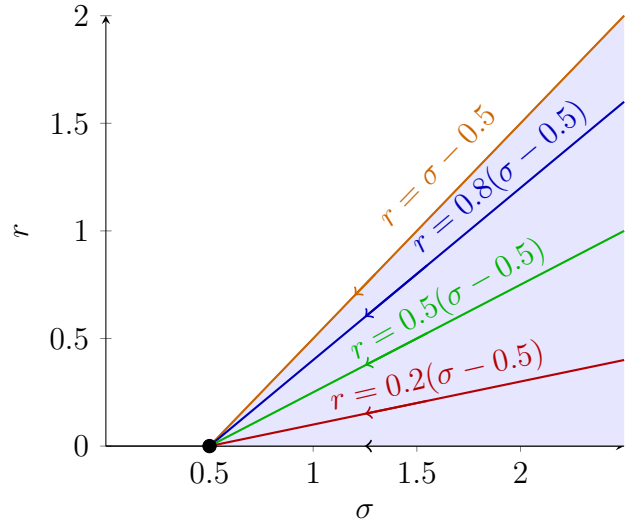
\begin{figure}[htbp]
\begin{center}
\begin{tikzpicture}

\begin{axis}[
    axis lines=left,
    xmin=0, xmax=2.5,
    ymin=0, ymax=2,
    xlabel={$\sigma$},
    ylabel={$r$},
    xtick={0.5,1,1.5,2},
    ytick={0,0.5,1,1.5,2},
    domain=0:2,
    samples=200,
    clip=false
]

\addplot[name path=upper, domain=0.5:2.5] {x - 0.5};
\addplot[name path=lower, domain=0:2.5] {0};
\addplot[fill=blue!10] fill between[of=upper and lower, soft clip={domain=0.5:2.5}];

\addplot[ thick, orange!80!black, domain=0.5:2.5] {x - 0.5};

\addplot[thick, blue!70!black, domain=0.5:2.5] {0.8*(x - 0.5)};
\addplot[thick, green!70!black, domain=0.5:2.5] {0.5*(x - 0.5)};
\addplot[thick, red!70!black, domain=0.5:2.5] {0.2*(x - 0.5)};

\node[orange!80!black, rotate=46]
    at (axis cs:1.6,1.3) {$r=\sigma-0.5$};
\node[blue!70!black, rotate=40]
    at (axis cs:1.9,1.25) {$r=0.8(\sigma-0.5)$};

\node[green!70!black, rotate=29]
    at (axis cs:1.65,0.68) {$r=0.5(\sigma-0.5)$};

\node[red!70!black, rotate=13]
    at (axis cs:1.6,0.32) {$r=0.2(\sigma-0.5)$};

\addplot[
    only marks,
    mark=*,
    mark size=2.5pt,
    black
] coordinates {(0.5,0)};

\draw[->, orange!80!black, thick]
    (axis cs:1.30,0.8) -- (axis cs:1.20,0.7);

\draw[->, blue!70!black, thick]
    (axis cs:1.55,0.84) -- (axis cs:1.25,0.60);

\draw[->, green!70!black, thick]
    (axis cs:1.55,0.525) -- (axis cs:1.25,0.375);

\draw[->, red!70!black, thick]
    (axis cs:1.55,0.21) -- (axis cs:1.25,0.15);
\draw[->, black!70!black, thick]
    (axis cs:1.26,0)-- (axis cs: 1.25,0);

\end{axis}

\end{tikzpicture}
\end{center}
\caption{Scaling lines approaching the boundary point $(\frac12,0)$.}
\end{figure}

\begin{corollary}\label{cor: affine}
Fix $\rho \in[0,1].$ For $\sigma>1/2$, take $r=\rho(\sigma-1/2)$ . Then
\[
(2\sigma-1)\,\|C_{\sigma + r{2^{-s}}}\|^2 \;\longrightarrow\; \|\widetilde{H}_{\rho/2}\|
\;=\;
\frac{2}{1 + \sqrt{1 - \rho^2}}\qquad \text{ as }\sigma\to \frac{1}{2}.
\]

\end{corollary}
\begin{proof}
For $\phi=\sigma+r2^{-s}$, observe that $\mathcal{H}_{\boldsymbol{\rho}}$ is simply equal to $\widetilde{H}_{{\rho}/2}$. Hence, the proof follows from Corollary \ref{cor} and Theorem \ref{Thm:HO}.
\end{proof}
Figure $1$ illustrates the different scaling paths $r=\rho(\sigma-1/2)$ along which the boundary point $(1/2,0)$ is approached. The limiting value depends on the choice of $\rho$, giving different limiting behavior along different paths.

The preceding corollary shows that equality holds in Lemma~\ref{lem:fp-boundedness-H-rho} whenever at most one component of $\boldsymbol{\rho}$ is nonzero. A natural question is whether any further equality cases can occur. The following theorem gives a complete
characterization of the equality cases.

\begin{theorem}
    For $\boldsymbol{\rho}\in B_d\setminus\{\boldsymbol{0}\},$ we have \[
\lambda(R_{\boldsymbol\rho})
-
\|\mathcal H_{\boldsymbol\rho}\|
\geq
\left(1-Q(\boldsymbol\rho)\right)
\left(\lambda(R_{\boldsymbol\rho})-1\right),
\] where
\[
Q(\boldsymbol\rho)
=
\frac{\sum_{j=1}^d\rho_j^2}
{\left(\sum_{j=1}^d\rho_j\right)^2}, \qquad \lambda(R_{\boldsymbol{\rho}})
=
\frac{2}{1+\sqrt{1-R_{\boldsymbol{\rho}}^2}}.
\]
Consequently, \[
\|\mathcal H_{\boldsymbol\rho}\|
=
\lambda(R_{\boldsymbol\rho})
\]
if and only if at most one component of \(\boldsymbol\rho\) is nonzero.
\end{theorem}

\begin{proof}
    Let $\delta>0.$ Then $C_{\varphi_{\delta,\boldsymbol\rho}}$ is bounded on $\mathcal{H}^2$. Applying  Theorem~5 in \cite{BP1} for symbol $\varphi_{\delta,\boldsymbol\rho}$ gives
\[
\|C_{\varphi_{\delta,\boldsymbol\rho}}f\|_{\mathcal H^2}^2
\leq
(1-Q(\boldsymbol\rho))
\left|
f\left(\frac12+\delta\right)
\right|^2
+
Q(\boldsymbol\rho)
\|C_{\psi_{\delta,R_{\boldsymbol\rho}}}f\|_{\mathcal H^2}^2,
\]
where \[
\psi_{\delta,R_{\boldsymbol{\rho}}}(s)
=
\frac12+\delta+\delta R_{\boldsymbol{\rho}}2^{-s}.
\]
By the Cauchy-Schwarz inequality, we have $$\sup_{\|f\|_{\mathcal{H}^2}=1}\left|
f\left(\frac12+\delta\right)
\right|^2\leq \zeta(1+2\delta).$$
Thus
\[
\|C_{\varphi_{\delta,\boldsymbol\rho}}\|_{\mathcal H^2}^2
\leq
(1-Q(\boldsymbol\rho))
\zeta(1+2\delta)
+
Q(\boldsymbol\rho)
\|C_{\psi_{\delta,R_{\boldsymbol\rho}}}\|_{\mathcal H^2}^2,
\]
In view of Corollary \ref{cor} and \ref{cor: affine}, multiplying  the above equation by $2\delta$ and letting   $\delta \to 0$, we obtain
\[
\|\mathcal H_{\boldsymbol\rho}\|
\leq
1+
Q(\boldsymbol\rho)
\left(
\lambda(R_{\boldsymbol\rho})-1
\right).
\]
Consequently, we have
\[
\lambda(R_{\boldsymbol\rho})
-
\|\mathcal H_{\boldsymbol\rho}\|
\geq
\left(1-Q(\boldsymbol\rho)\right)
\left(\lambda(R_{\boldsymbol\rho})-1\right).
\]
Suppose that equality holds in the limiting norm estimate. Since
$R_{\boldsymbol\rho}>0$, we have
$\lambda(R_{\boldsymbol\rho})>1$, and hence
\[
Q(\boldsymbol\rho)=1.
\]
Equivalently,
\[
\sum_{j=1}^d\rho_j^2
=
\left(\sum_{j=1}^d\rho_j\right)^2.
\]
Since $\rho_j\ge0$, this holds if and only if at most one component of
$\boldsymbol\rho$ is nonzero. Conversely, if at most one component of
$\boldsymbol\rho$ is nonzero,
by  Proposition~\ref{prop:fp-coefficient-realization}, the norm of a one-prime coefficient model is independent of the chosen prime. After
relabelling the active coordinate, Corollary~\ref{cor: affine} therefore
gives equality.
\end{proof}

The preceding theorem characterizes the equality cases in the limiting
norm estimate. We now describe the structure of
$\mathcal H_{\boldsymbol\rho}$ according to total degree.

\begin{lemma}
\label{lem:boundary-layer-comparison}
Let $N$ be a non--negative integer. For $\alpha=(\alpha_1,\ldots,\alpha_d)\in\mathbb{N}_{0}^{d}$ and $\boldsymbol{\rho}\in B_d$, define
\[
h_N(\boldsymbol{\rho})^{2}
=
\sum_{|\alpha|=N}
\frac{1}{(\alpha!)^{2}}
\prod_{j=1}^{d}
\left(\frac{\rho_j}{2}\right)^{2\alpha_j}.
\] Then
\begin{equation} \label{eq:h-rho inequality}
h_N(\boldsymbol{\rho})
\leq
\frac{(R_{\boldsymbol{\rho}}/2)^N}{N!}.
\end{equation}
For $N\geq1$, equality holds in \eqref{eq:h-rho inequality} if and only if at most one of the $\rho_j$ is nonzero.
\end{lemma}
\begin{proof}
By the multinomial theorem,
\[
(R_{\boldsymbol{\rho}}/2)^N
=
(\rho_1/2+\cdots+\rho_d/2)^N
=
\sum_{|\alpha|=N}
\frac{N!}{\alpha!}(\boldsymbol{\rho}/2)^{\alpha}.
\]
Therefore
\[
\frac{(R_{\boldsymbol{\rho}}/2)^N}{N!}
=
\sum_{|\alpha|=N}
\frac{(\boldsymbol{\rho}/2)^{\alpha}}{\alpha!}.
\]
All terms in this sum are nonnegative. Hence
\[
\sum_{|\alpha|=N}
\left(\frac{(\boldsymbol{\rho}/2)^{\alpha}}{\alpha!}\right)^2
\leq
\left(
\sum_{|\alpha|=N}
\frac{(\boldsymbol{\rho}/2)^{\alpha}}{\alpha!}
\right)^2
= \left(\frac{(R_{\boldsymbol{\rho}}/2)^N}{N!}\right)^2.
\]
This is exactly
\[
h_N(\boldsymbol{\rho})^{2}
\leq
\frac{(R_{\boldsymbol{\rho}}/2)^{2N}}{(N!)^{2}}.
\]
If at most one of the \(\rho_j\)'s is nonzero, then there is at most one multi-index \(\alpha\) with $|\alpha|=N$  and
$\boldsymbol{\rho}^{\alpha}\neq 0.$ In that case, we get
\[
h_N(\boldsymbol{\rho})
=
\frac{(R_{\boldsymbol{\rho}}/2)^N}{N!}
\]
Conversely, assume that at least two of the \(\rho_j\)'s are strictly positive. Then, for every $N\geq 1,$ there are at least two multi-indices \(\alpha\) with $|\alpha|=N$ and
$\boldsymbol{\rho}^{\alpha}>0.$

For instance, if \(\rho_i>0\) and \(\rho_j>0\) with \(i\neq j\), then the multi-indices $Ne_i  \text{ and } Ne_j$
both contribute positively. Therefore the nonnegative numbers
\[
\frac{(\boldsymbol{\rho}/2)^{\alpha}}{\alpha!},
\qquad
|\alpha|=N,
\]
contain at least two strictly positive entries. For any finite family of nonnegative numbers with at least two positive entries, one has
\[
\sum q_{\alpha}^{2}
<
\left(\sum q_{\alpha}\right)^{2}.
\]
Applying this to
$q_{\alpha}=\frac{(\boldsymbol{\rho}/2)^{\alpha}}{\alpha!}$
gives
\[
h_N(\boldsymbol{\rho})^{2}
=
\sum_{|\alpha|=N}
\left(\frac{(\boldsymbol{\rho}/2)^{\alpha}}{\alpha!}\right)^2
<
\left(
\sum_{|\alpha|=N}
\frac{(\boldsymbol{\rho}/2)^{\alpha}}{\alpha!}
\right)^2
=
\frac{(R_{\boldsymbol{\rho}}/2)^{2N}}{(N!)^{2}}
\qquad
\text{for }N\geq 1.
\]

\end{proof}
For $\boldsymbol{\rho}\in B_d\setminus\{\mathbf0\}$, the above
estimate suggests the normalization
\[
c_N(\boldsymbol{\rho})
:=
\frac{h_N(\boldsymbol{\rho})}
{(R_{\boldsymbol{\rho}}/2)^N/N!},
\qquad N\in\mathbb N_0.
\]
Define the diagonal operator $D_{\boldsymbol{\rho}}$ on
$\ell^2(\mathbb N_0)$ by
\[
D_{\boldsymbol{\rho}}e_N
=
c_N(\boldsymbol{\rho})e_N,
\qquad N\in\mathbb N_0.
\]
By Lemma \ref{lem:boundary-layer-comparison}, $0<c_N(\boldsymbol{\rho})\le1,$
and hence $D_{\boldsymbol{\rho}}$ is a positive contraction and,
in particular, self-adjoint.
The following theorem provides a structural decomposition of $\mathcal{H}_{\boldsymbol{\rho}}$, reducing its analysis to that of a one-dimensional weighted Hankel operator.
\begin{theorem}
\label{thm:boundary-Hankel-reduction}
For $\boldsymbol{\rho}\in B_d$, the operator  \(\mathcal{H}_{\boldsymbol{\rho}}\) is unitarily equivalent to $\widetilde{B}_{\boldsymbol{\rho}}\oplus \mathbf{0},$ where
\[
\widetilde{B}_{\boldsymbol{\rho}}
=
\left(
(-1)^{N+M}(N+M)!h_N(\boldsymbol{\rho})h_M(\boldsymbol{\rho})
\right)_{N,M\geq 0}.
\]
Moreover, if $\boldsymbol{\rho}\neq \mathbf{0}$ then
$B_{\boldsymbol{\rho}} =
D_{\boldsymbol{\rho}}H_{\frac{R_{\boldsymbol{\rho}}}{2}}D_{\boldsymbol{\rho}},$
where $B_{\boldsymbol\rho}$ is unitarily equivalent to  $\widetilde B_{\boldsymbol\rho}$ by a diagonal unitary operator.

\end{theorem}

\begin{proof}
For $\alpha\in\mathbb{N}_{0}^{d},$ set
\[
w_{\alpha}
=
\frac{1}{\alpha!}
\prod_{j=1}^{d}
\left(\frac{\rho_j}{2}\right)^{\alpha_j}.
\]
Then
\[
h_N(\boldsymbol{\rho})^{2}
=
\sum_{|\alpha|=N}w_{\alpha}^{2}.
\]
The case  \(R_{\boldsymbol{\rho}}=0\)  is immediate, so assume henceforth that $R_{\boldsymbol{\rho}}>0$. Then,  for every \(N\) there is at least one multi-index \(\alpha\) with $|\alpha|=N$ and $w_{\alpha}>0.$
Thus $h_N(\boldsymbol{\rho})>0.$ Define
$\omega_N\in\ell^{2}(\mathbb{N}_{0}^{d})$ by
\[
(\omega_N)_{\alpha}
=
\begin{cases}
\displaystyle \frac{w_{\alpha}}{h_N(\boldsymbol{\rho})}, & |\alpha|=N,\\[6pt]
0, & |\alpha|\neq N.
\end{cases}
\]
The vectors $\omega_0,\omega_1,\omega_2,\ldots$ are orthonormal because they have disjoint total-degree supports and each has norm \(1\). Let $\mathcal{M}_{\boldsymbol{\rho}}=\overline{\operatorname{span}}\{\omega_N:N\geq 0\}$.

For $|\alpha|=N$ and $|\beta|=M,$ the matrix entry of \(\mathcal{H}_{\boldsymbol{\rho}}\) can be written as
\[
(\mathcal{H}_{\boldsymbol{\rho}})_{\alpha,\beta}
=
(-1)^{N+M}(N+M)!w_{\alpha}w_{\beta}.
\]

Let $x=(x_{\beta})_{\beta\in\mathbb{N}_{0}^{d}}$ be finitely supported. For $|\alpha|=N,$ we have
\[
\begin{aligned}
(\mathcal{H}_{\boldsymbol{\rho}}x)_{\alpha}
&=
\sum_{\beta\in\mathbb{N}_{0}^{d}}
(\mathcal{H}_{\boldsymbol{\rho}})_{\alpha,\beta}x_{\beta}  \\
&=
(-1)^Nw_{\alpha}
\sum_{M=0}^{\infty}
(-1)^M(N+M)!
\sum_{|\beta|=M}w_{\beta}x_{\beta}.
\end{aligned}
\]
Consequently, for each fixed \(N\),
\[
\sum_{|\alpha|=N}(\mathcal{H}_{\boldsymbol{\rho}}x)_{\alpha}e_{\alpha}= (-1)^N
\sum_{M=0}^{\infty}
(-1)^M(N+M)!
\sum_{|\beta|=M}w_{\beta}x_{\beta} \sum_{|\alpha|=N}w_{\alpha}e_{\alpha}
\]
Since $\sum_{|\alpha|=N}w_{\alpha}e_{\alpha}
= h_N(\boldsymbol{\rho})\omega_N,$
\[
\mathcal{H}_{\boldsymbol{\rho}}x =  \sum_{N=0}^{\infty}\left (
\sum_{M=0}^{\infty}
(-1)^{(N+M)}(N+M)!
\sum_{|\beta|=M}w_{\beta}x_{\beta} \right  )h_N(\boldsymbol{\rho})\omega_N
\]
Therefore, $\mathcal{H}_{\boldsymbol{\rho}}x\in \mathcal{M}_{\boldsymbol{\rho}}$  for every finitely supported \(x\). Since \(\mathcal{H}_{\boldsymbol{\rho}}\) is bounded (from Lemma \ref{lem:fp-boundedness-H-rho}), it follows that
$\operatorname{Ran}\mathcal{H}_{\boldsymbol{\rho}}
\subset \mathcal{M}_{\boldsymbol{\rho}}.$

Now let $x\in\mathcal{M}_{\boldsymbol{\rho}}^{\perp}.$
Then, for every $M\in\mathbb{N}_{0},$ we have
\[
0
=
\left\langle x,\omega_M\right\rangle
=
\frac{1}{h_M(\boldsymbol{\rho})}
\sum_{|\beta|=M}x_{\beta}w_{\beta}.
\]
Hence
\[
\sum_{|\beta|=M}w_{\beta}x_{\beta}=0
\qquad
\text{for every}
\qquad
M\in\mathbb{N}_{0}.
\]
Applying the preceding formula  for $\mathcal{H}_{\boldsymbol{\rho}}$ first to finitely supported vectors and then using density, we obtain

\[
\mathcal{H}_{\boldsymbol{\rho}}x=0.
\]
Thus $\mathcal{H}_{\boldsymbol{\rho}}$ annihilates
$\mathcal{M}_{\boldsymbol{\rho}}^{\perp}.$

Define $U:\ell^2(\mathbb{N}_{0})\to \mathcal{M}_{\boldsymbol{\rho}}$  by
\[
Ue_N=\omega_N,
\qquad
N\in\mathbb{N}_{0}.
\]
Since the vectors $\omega_N$ form an orthonormal basis for
$\mathcal{M}_{\boldsymbol{\rho}},$ the map \(U\) is unitary.

We compute the matrix of $U^{*}\mathcal{H}_{\boldsymbol{\rho}}U.$ For $|\alpha|=N$ and
$M\in\mathbb{N}_{0},$ we have
\[
\begin{aligned}
(\mathcal{H}_{\boldsymbol{\rho}}\omega_M)_{\alpha}=\sum_{\beta\in\mathbb{N}_0^d}(\mathcal{H}_{\boldsymbol{\rho}})_{\alpha,\beta}(\omega_M)_{\beta}
&=
\sum_{|\beta|=M}
(-1)^{N+M}(N+M)!w_{\alpha}w_{\beta}
\frac{w_{\beta}}{h_M(\boldsymbol{\rho})}                                      \\
&=
(-1)^{N+M}(N+M)!w_{\alpha}
\frac{1}{h_M(\boldsymbol{\rho})}
\sum_{|\beta|=M}w_{\beta}^{2}                              \\
&=
(-1)^{N+M}(N+M)!w_{\alpha}h_M(\boldsymbol{\rho}).
\end{aligned}
\]
Since $w_{\alpha}=h_N(\boldsymbol{\rho})(\omega_N)_{\alpha}$ for every coordinate $\alpha$ with $|\alpha|=N$,
\[
(\mathcal{H}_{\boldsymbol{\rho}}\omega_M)_{\alpha}= (-1)^{N+M}(N+M)!h_N(\boldsymbol{\rho)}h_M(\boldsymbol{\rho})(\omega_N)_{\alpha}.
\]
Then
\[
\mathcal{H}_{\boldsymbol{\rho}}\omega_M
= \sum_{N=0}^{\infty}\sum_{|\alpha|=N}(\mathcal{H}_{\boldsymbol{\rho}}\omega_M)_{\alpha}e_{\alpha}=
\sum_{N=0}^{\infty}
(-1)^{N+M}
(N+M)!h_N(\boldsymbol{\rho})h_M(\boldsymbol{\rho})
\omega_N.
\]
Equivalently, $U^{*}\mathcal{H}_{\boldsymbol{\rho}}U=\widetilde{B}_{\boldsymbol{\rho}}$, where
\[
\widetilde{B}_{\boldsymbol{\rho}}
=
\left(
(-1)^{N+M}(N+M)!h_N(\boldsymbol{\rho})h_M(\boldsymbol{\rho})
\right)_{N,M\geq 0}.
\]
Let $U_1:\ell^{2}(\mathbb{N}_{0})\to\ell^{2}(\mathbb{N}_{0})$
be the diagonal unitary $(U_1x)_N=(-1)^Nx_N.$ Then
\[
U_1\widetilde{B}_{\boldsymbol{\rho}}U_1
=
B_{\boldsymbol{\rho}},
\]
where
\[
(B_{\boldsymbol{\rho}})_{N,M}
=
(N+M)!h_N(\boldsymbol{\rho})h_M(\boldsymbol{\rho}).
\]
For $\boldsymbol{\rho}\neq \mathbf{0}$, observe that
\[
(N+M)!h_N(\boldsymbol{\rho})h_M(\boldsymbol{\rho})
=
c_N(\boldsymbol{\rho})c_M(\boldsymbol{\rho})
\binom{N+M}{N}\left(\frac{R_{\boldsymbol{\rho}}}{2}\right)^{N+M}.
\]
Hence
$B_{\boldsymbol{\rho}}
= D_{\boldsymbol{\rho}}H_{\frac{R_{\boldsymbol{\rho}}}{2}}D_{\boldsymbol{\rho}}.$

\end{proof}
The decomposition above isolates the weighted Hankel operator
$H_{R_{\boldsymbol{\rho}}/2}$ from the contribution of the prime
coefficients, which is encoded in the diagonal sequence
$\{c_N(\boldsymbol{\rho})\}_{N\ge0}$. We next give a concrete
representation of these diagonal coefficients and analyze their
behavior.

\begin{proposition}\label{prop:conv}
Let \(\boldsymbol\rho\in B_d\setminus\{\boldsymbol0\}\), and set
\[
\nu_{\boldsymbol\rho}
=
\sum_{j=1}^d
\frac{\rho_j}{R_{\boldsymbol\rho}}\,\delta_{e_j}.
\]
With \(\nu_{\boldsymbol\rho}^{*0}=\delta_0\), we have
\[
c_N(\boldsymbol\rho)
=
\|\nu_{\boldsymbol\rho}^{*N}\|_{\ell^2(\mathbb Z^d)},
\qquad N\geq0.
\]
If at least two components of \(\boldsymbol\rho\) are positive, then
\[
c_N(\boldsymbol\rho)\longrightarrow0 \qquad \text{ as } N \longrightarrow \infty .
\]
\end{proposition}

\begin{proof}
Set
$\theta_j = \frac{\rho_j}{R_{\boldsymbol{\rho}}},$  for $ 1\leq j\leq d.$
Then $\sum_{j=1}^d\theta_j=1$ and $c_N^2$ can be written as
\[
c_N^2
=
\sum_{|\alpha|=N}
\left(
\frac{N!}{\alpha!}
\theta_1^{\alpha_1}\cdots\theta_d^{\alpha_d}
\right)^2.
\]

Let $e_1,\ldots,e_d$ denote the standard basis of $\mathbb R^d$, and define the finitely supported function $\nu:\mathbb Z^d\to\mathbb R$ by
\[
\nu:=\sum_{j=1}^d\theta_j\delta_{e_j},
\]
where $\delta_{e_j}$ denotes the point mass at $e_j$. For functions
$f,g$ on $\mathbb Z^d$, their convolution is defined by
\[
(f*g)(\alpha)
=
\sum_{\beta\in\mathbb Z^d}
f(\beta)g(\alpha-\beta).
\]
Accordingly,
\[
\nu^{*N}
=
\underbrace{\nu*\cdots*\nu}_{N\text{ times}}
\]
and, for $\alpha\in\mathbb Z^d$,
\[
\nu^{*N}(\alpha)
=
\sum_{\substack{\beta_1,\ldots,\beta_N\in\mathbb Z^d\\
\beta_1+\cdots+\beta_N=\alpha}}
\nu(\beta_1)\cdots\nu(\beta_N).
\]

Since $\nu$ is supported on $\{e_1,\ldots,e_d\}$, only those terms
for which each $\beta_k$ belongs to $\{e_1,\ldots,e_d\}$ contribute to
the above sum. Hence
\[
\nu^{*N}(\alpha)
=
\sum_{\substack{
j_1,\ldots,j_N\in\{1,\ldots,d\}\\
e_{j_1}+\cdots+e_{j_N}=\alpha
}}
\theta_{j_1}\cdots\theta_{j_N}.
\]
Now let $\alpha=(\alpha_1,\ldots,\alpha_d)\in\mathbb N_0^d$ with
$|\alpha|=N$. The identity $e_{j_1}+\cdots+e_{j_N}=\alpha$
means that $e_j$ occurs exactly $\alpha_j$ times among
$e_{j_1},\ldots,e_{j_N}$, for each $1\leq j\leq d$. Consequently,
every nonzero term in the above sum equals
\[
\theta_1^{\alpha_1}\cdots\theta_d^{\alpha_d}
=
\boldsymbol{\theta}^\alpha.
\]
Moreover, the number of ordered $N$-tuples
$(j_1,\ldots,j_N)$ with this property is
\[
\frac{N!}{\alpha_1!\cdots\alpha_d!}
=
\frac{N!}{\alpha!}.
\]
Therefore,
\[
\nu^{*N}(\alpha)
=
\frac{N!}{\alpha!}\boldsymbol{\theta}^\alpha,
\qquad |\alpha|=N.
\]
Also,
\[
\operatorname{supp}(\nu^{*N})
\subseteq
\{\alpha\in\mathbb N_0^d:|\alpha|=N\}.
\]
It follows that
\begin{equation}\label{eq:convolution}
c_N^2(\boldsymbol{\rho})
=
\sum_{|\alpha|=N}
\left|
\frac{N!}{\alpha!}\boldsymbol{\theta}^\alpha
\right|^2
=
\sum_{\alpha\in\mathbb Z^d}
|\nu^{*N}(\alpha)|^2
=
\|\nu^{*N}\|_{\ell^2(\mathbb Z^d)}^2.
\end{equation}

For $\mathbf t=(t_1,\ldots,t_d)\in\mathbb T^d$, let
$\widehat{\nu}$ denote the Fourier transform of $\nu$, defined by
\[
\widehat{\nu}(\mathbf t)
:=
\sum_{\alpha\in\mathbb Z^d}
\nu(\alpha)e^{i\alpha\cdot\mathbf t}, \qquad
 \text{ where } \alpha\cdot\mathbf t
= \alpha_1t_1+\cdots+\alpha_dt_d.\]
Since $\nu$ is supported on $\{e_1,\ldots,e_d\}$, we have
\[
\widehat{\nu}(\mathbf t)
=
\sum_{j=1}^d\theta_j e^{it_j}.
\]
The Fourier transform converts convolution into multiplication, and
hence
\[
\widehat{\nu^{*N}}(\mathbf t)
=
\widehat{\nu}(\mathbf t)^N
=
\left(
\sum_{j=1}^d\theta_j e^{it_j}
\right)^N.
\]
Therefore, by Plancherel's identity on $\mathbb Z^d$,
\[
c_N^2(\boldsymbol{\rho})
=
\|\nu^{*N}\|_{\ell^2(\mathbb Z^d)}^2
=
\int_{\mathbb T^d}
|\widehat{\nu^{*N}}(\mathbf t)|^2
\,d\mathbf m(\mathbf t)
=
\int_{\mathbb T^d}
\left|
\sum_{j=1}^d\theta_j e^{it_j}
\right|^{2N}
\,d\mathbf m(\mathbf t),
\]
where $\mathbf m$ denotes the normalized Haar measure on
$\mathbb T^d$.

Since at least two $\theta_j$ are strictly positive, the modulus inside the integral is strictly smaller than $1$ outside the measure-zero set on which all active phases coincide.
 By the dominated convergence theorem
\[
\lim_{N\to\infty}c_N^2(\boldsymbol{\rho})=0.
\]

\end{proof}
We now record the consequences for the
diagonal operator $D_{\boldsymbol{\rho}}$ and, hence, for the limiting
operator $\mathcal H_{\boldsymbol{\rho}}$.

If at least two components of $\boldsymbol{\rho}$ are positive, then
 Proposition \ref{prop:conv} gives
\[
c_N(\boldsymbol{\rho})\longrightarrow0.
\]
Hence
\[
D_{\boldsymbol{\rho}}
=
\operatorname{diag}
\bigl(c_0(\boldsymbol{\rho}),
c_1(\boldsymbol{\rho}),\ldots\bigr)
\]
is compact. Therefore, by Theorem~\ref{thm:boundary-Hankel-reduction},
$\mathcal H_{\boldsymbol{\rho}}$
is compact as well.

\section{Fixed-\texorpdfstring{$\sigma$}{sigma} asymptotic expansion}
\label{sec:4}

We now turn to the fixed-$\sigma$ case for finite-prime symbols. Let
$$\varphi(s)=\sigma+\sum_{j=1}^{d}r_jp_j^{-s},$$ with
 \begin{equation}\label{eq:fixed-sigma}
    \sigma>\frac12,\quad r_j\geq0, \quad R:=\sum_{j=1}^{d}r_j<\sigma-1/2.
\end{equation}

The purpose of this section is to derive the asymptotic expansion of $\|C_{\varphi}\|^2$ as $R\to 0$. Our approach is based on a rank-one decomposition of the positive operator
$S_{\sigma,\boldsymbol{r}}=T_{\sigma,\boldsymbol{r}}^{*}T_{\sigma,\boldsymbol{r}}$
followed by an estimate of the higher-order terms and a reduction to a rank-two approximation. For $\alpha\in\mathbb{N}_{0}^{d},$ define $(u_{\alpha})_n$ by
\[
(u_{\alpha})_n
=
n^{-\sigma}
\frac{(\log n)^{|\alpha|}}{\alpha!},
\qquad
n\in\mathbb{N}.
\]

\begin{lemma}
\label{lem:finite-prime-input-rank-one-expansion}
Let $\sigma, \boldsymbol{r}$ be as in \eqref{eq:fixed-sigma}. Then $((u_{\alpha})_n)\in\ell^{2}(\mathbb{N})$, and
\[
S_{\sigma,\boldsymbol{r}}
=
\sum_{\alpha\in\mathbb{N}_{0}^{d}}
\boldsymbol{r}^{2\alpha}(u_{\alpha}\otimes u_{\alpha}),
\]
where the series converges absolutely in operator norm, and  $(u_{\alpha} \otimes u_{\alpha} )a = \langle a, u_{\alpha} \rangle u_{\alpha}$.
\end{lemma}

\begin{proof}
For $a \in \ell^2(\mathbb{N})$, the definition of $T_{\sigma, \boldsymbol{r}}$ gives
\[
(T_{\sigma,\boldsymbol{r}}a)_{\alpha}
=
\boldsymbol{r}^{\alpha}
\sum_{n=1}^{\infty}
a_n n^{-\sigma}
\frac{(-\log n)^{|\alpha|}}{\alpha!}
= (-1)^{|\alpha|}\boldsymbol{r}^{\alpha}\langle a,u_{\alpha}\rangle.
\]
Then one can easily compute its adjoint and verify,
\[
S_{\sigma,\boldsymbol{r}}a
=
\sum_{\alpha\in\mathbb{N}_{0}^{d}}
\boldsymbol{r}^{2\alpha}\langle a,u_{\alpha}\rangle u_{\alpha}=\sum_{\alpha\in\mathbb{N}_{0}^{d}}
\boldsymbol{r}^{2\alpha}(u_{\alpha}\otimes u_{\alpha}).
\]
It remains to prove that this series converges absolutely in the operator norm.
Fix $\alpha\in \mathbb{N}_0^d$, put $m=|\alpha|.$ Then
\[
\|u_{\alpha}\|^{2}
=
\frac{1}{(\alpha!)^{2}}
\sum_{n=1}^{\infty}n^{-2\sigma}(\log n)^{2m}.
\]
Since $R<\sigma-\frac12,$ we may choose $\varepsilon$ such that
$2R<\varepsilon<2\sigma-1.$
From the exponential series,
\[
e^{\varepsilon\log n}
\ge
\frac{(\varepsilon\log n)^{2m}}{(2m)!},
\]
and therefore
\begin{equation}\label{eq:exp-inequality}
(\log n)^{2m}
\le
\frac{(2m)!}{\varepsilon^{2m}}\,n^\varepsilon.
\end{equation}
It follows that
\[
\|u_{\alpha}\|^{2}
\leq
\frac{(2m)!}{(\alpha!)^{2}\varepsilon^{2m}}
\zeta(2\sigma-\varepsilon).
\]
Thus
\[
\sum_{|\alpha|=m}
\boldsymbol{r}^{2\alpha}\|u_{\alpha}\|^{2}
\leq
\zeta(2\sigma-\varepsilon)
\frac{(2m)!}{\varepsilon^{2m}}
\sum_{|\alpha|=m}
\frac{\boldsymbol{r}^{2\alpha}}{(\alpha!)^{2}}.
\]
By the multinomial theorem,
\[
R^m
=
(r_1+\cdots+r_d)^m
=
\sum_{|\alpha|=m}
\frac{m!}{\alpha!}\boldsymbol{r}^{\alpha}.
\]
Therefore
\[
\frac{R^m}{m!}
=
\sum_{|\alpha|=m}
\frac{\boldsymbol{r}^{\alpha}}{\alpha!}.
\]
All terms in this sum are nonnegative. Hence
\[
\sum_{|\alpha|=m}
\left(\frac{\boldsymbol{r}^{\alpha}}{\alpha!}\right)^2
\leq
\left(
\sum_{|\alpha|=m}
\frac{\boldsymbol{r}^{\alpha}}{\alpha!}
\right)^2
=
\left(\frac{R^m}{m!}\right)^2.
\]
Thus
\[
\sum_{|\alpha|=m}
\boldsymbol{r}^{2\alpha}\|u_{\alpha}\|^{2}
\leq
\zeta(2\sigma-\varepsilon)
\binom{2m}{m}
\left(\frac{R^2}{\varepsilon^2}\right)^m.
\]
Using $\binom{2m}{m}\leq 4^m,$ we get
\begin{equation} \label{eq: summation bound}
\sum_{|\alpha|=m}
\boldsymbol{r}^{2\alpha}\|u_{\alpha}\|^{2}
\leq
\zeta(2\sigma-\varepsilon)
\left(\frac{4R^2}{\varepsilon^2}\right)^m<\infty.
\end{equation}
Observe that for any $a\in \ell^2$,
\[
\|(u_{\alpha}\otimes u_{\alpha})(a)\|=\|<a,u_{\alpha}>u_{\alpha}\|\leq \|a\|\|u_{\alpha}\|^2
\]
and for $a=\frac{u_{\alpha}}{\|u_{\alpha}\|}$,
\[
\|(u_{\alpha}\otimes u_{\alpha})(a)\|=\|a\|\|u_{\alpha}\|^2.
\]
Hence $\|u_{\alpha} \otimes u_{\alpha}\| = \|u_{\alpha}\|^2$. Using this, we obtain

\[
\sum_{\alpha\in\mathbb{N}_{0}^{d}}
\boldsymbol{r}^{2\alpha}\|u_{\alpha}\otimes u_{\alpha}\|
=
\sum_{\alpha\in\mathbb{N}_{0}^{d}}
\boldsymbol{r}^{2\alpha}\|u_{\alpha}\|^{2}
<\infty.
\]
\end{proof}

\begin{lemma}
\label{lem:finite-prime-tail-beyond-degree-one}
As $R\to 0,$  we have
\[
\left\|
\sum_{|\alpha|\geq 2}
\boldsymbol{r}^{2\alpha}(u_{\alpha}\otimes u_{\alpha})
\right\|
= O(R^4).
\]
\end{lemma}

\begin{proof}
Fix $\sigma
>1/2$ and set $\varepsilon=\frac{2\sigma-1}{2}$  and $R_0=\frac{\varepsilon}{4}.$ Then, for  $0\leq R\leq R_0,$   $q:=\frac{4R^2}{\varepsilon^2}\leq \frac{1}{4}.$

By the estimate established in Lemma  \ref{lem:finite-prime-input-rank-one-expansion}, for every $m\ge2$,
\[
\sum_{|\alpha|=m}
\boldsymbol{r}^{2\alpha}\|u_{\alpha}\otimes u_{\alpha}\|
\leq
\zeta(2\sigma-\varepsilon)
\left(\frac{4R^2}{\varepsilon^2}\right)^m.
\]
Therefore
\[
\left\|
\sum_{|\alpha|\geq 2}
\boldsymbol{r}^{2\alpha}(u_{\alpha}\otimes u_{\alpha})
\right\|
\leq
\zeta(2\sigma-\varepsilon)
\sum_{m=2}^{\infty}
\left(\frac{4R^2}{\varepsilon^2}\right)^m.
\]
Since $q\leq \frac{1}{4},$
we have $ \sum_{m=2}^{\infty}q^m = \frac{q^2}{1-q} \leq \frac{4}{3}q^2.$
Thus
\[
\left\|
\sum_{|\alpha|\geq 2}
\boldsymbol{r}^{2\alpha}(u_{\alpha}\otimes u_{\alpha})
\right\|
\leq
\frac{4}{3}\zeta(2\sigma-\varepsilon)
\left(\frac{4R^2}{\varepsilon^2}\right)^2=CR^4,
\]
where
$C = \frac{64}{3}\frac{\zeta(2\sigma-\varepsilon)}{\varepsilon^4}$
is independent of $R$. Hence the desired $O(R^4)$ estimate follows.

\end{proof}

\begin{lemma}
\label{lem:finite-prime-rank-two-calculation}
Let $\sigma,\boldsymbol{r}$ be as in \eqref{eq:fixed-sigma}.
For $u_0=(n^{-\sigma})_{n\geq 1},u_1=((\log n)n^{-\sigma})_{n\geq 1}$  and $Q=\sum_{j=1}^{d}r_j^2,$ define
\[
S_{\sigma,\boldsymbol{r}}^{(2)}
=
(u_0\otimes u_0)+Q(u_1\otimes u_1).
\]
Then
\[
\|S_{\sigma,\boldsymbol{r}}^{(2)}\|
=
\zeta(2\sigma)
+
\frac{\zeta'(2\sigma)^2}{\zeta(2\sigma)}Q
+
O(Q^2) \qquad \text{ as } Q\to 0.
\]
\end{lemma}

\begin{proof}
Define $B:\mathbb{C}^{2}\to\ell^{2}(\mathbb{N})$ by
\[
B(\xi_0,\xi_1)=\xi_0u_0+\xi_1Q^{1/2}u_1.
\]
Since $u_0,u_1 \in \ell^2(\mathbb{N})$, the operator $B$ is bounded. A direct computation shows that
 $S_{\sigma,\boldsymbol{r}}^{(2)}=BB^{*}.$

The matrix  $B^*B$ with respect to  the standard basis of $\mathbb{C}^2$ is
\[
B^{*}B
=
\begin{pmatrix}
\langle u_0,u_0\rangle & Q^{1/2}\langle u_1,u_0\rangle\\
Q^{1/2}\langle u_0,u_1\rangle & Q\langle u_1,u_1\rangle
\end{pmatrix}.
\]
Next we find out the largest eigenvalue of $B^{*}B$, that gives the norm of $S_{\sigma,\boldsymbol{r}}^{(2)}$.

Since $$\langle u_0,u_0\rangle=\zeta(2\sigma),
\langle u_1,u_0\rangle
=
\sum_{n=1}^{\infty}n^{-2\sigma}\log n
=
-\zeta'(2\sigma),$$
and $$
\langle u_1,u_1\rangle
=
\sum_{n=1}^{\infty}n^{-2\sigma}(\log n)^2
=
\zeta''(2\sigma).
$$
Set
\[
A=\zeta(2\sigma),
\qquad
B_1=-\zeta'(2\sigma),
\qquad
C=\zeta''(2\sigma).
\]
Then the relevant \(2\times 2\) matrix is
\[
\begin{pmatrix}
A & Q^{1/2}B_1\\
Q^{1/2}B_1 & QC
\end{pmatrix}.
\]
Its larger eigenvalue is
\[
\frac{A+QC+\sqrt{(A-QC)^2+4QB_1^2}}{2}.
\]

Since
\[
(A-QC)^2+4QB_1^2
=
A^2+Q(-2AC+4B_1^2)+C^2Q^2,
\]
we obtain
\[
\sqrt{(A-QC)^2+4QB_1^2}
=
A\sqrt{1+
Q\frac{-2AC+4B_1^2}{A^2}
+\frac{C^2}{A^2}Q^2}.
\]
Using the Taylor expansion
\[
\sqrt{1+x}=1+\frac{x}{2}+O(x^2), \qquad x\to0,
\]
and observing that
\[
Q\frac{-2AC+4B_1^2}{A^2}+\frac{C^2}{A^2}Q^2=O(Q),
\]
it follows that
\[
\sqrt{(A-QC)^2+4QB_1^2}
=
A+
\left(
\frac{2B_1^2}{A}-C
\right)Q
+O(Q^2).
\]
Therefore the larger eigenvalue is
\[
A+\frac{B_1^2}{A}Q+O(Q^2).
\]
Substituting back the value of A, B, C gives
\[
\|S_{\sigma,\boldsymbol{r}}^{(2)}\|
=
\zeta(2\sigma)
+
\frac{\zeta'(2\sigma)^2}{\zeta(2\sigma)}Q
+
O(Q^2).
\]
\end{proof}

\begin{theorem}
\label{T:finite-prime-fixed-sigma-expansion}
Let $\varphi=\sigma+\sum_{j=1}^{d}r_jp_j^{-s}$ be the finite-prime symbol  satisfying  \eqref{eq:fixed-sigma}. Then as $R\to 0,$ the squared norm of the associated composition operator admits the expansion
\[
\left\|C_{\sigma+\sum_{j=1}^{d}r_jp_j^{-s}}\right\|^{2}
=
\zeta(2\sigma)
+
\frac{\zeta'(2\sigma)^2}{\zeta(2\sigma)}
\sum_{j=1}^{d}r_j^2
+
O(R^4),
\]
with $\sigma>1/2$ fixed.
\end{theorem}

\begin{proof}
From Lemma \ref{lem:finite-prime-input-rank-one-expansion} and \ref{lem:finite-prime-rank-two-calculation}, we can write
\[
S_{\sigma,\boldsymbol{r}}
=
S_{\sigma,\boldsymbol{r}}^{(2)}
+
\sum_{|\alpha|\geq 2}
\boldsymbol{r}^{2\alpha}(u_{\alpha}\otimes u_{\alpha}),
\]
By Lemma \ref{lem:finite-prime-tail-beyond-degree-one},
\[
\left\|
S_{\sigma,\boldsymbol{r}}-S_{\sigma,\boldsymbol{r}}^{(2)}
\right\|
=
O(R^4).
\]
Hence
\[
\left|
\|S_{\sigma,\boldsymbol{r}}\|
-
\|S_{\sigma,\boldsymbol{r}}^{(2)}\|
\right|
=
O(R^4).
\]
Since \(Q\leq R^2\), we have \(Q^2\leq R^4\).  Lemma \ref{lem:finite-prime-rank-two-calculation} gives
\[
\|S_{\sigma,\boldsymbol{r}}\|
=
\zeta(2\sigma)
+
\frac{\zeta'(2\sigma)^2}{\zeta(2\sigma)}
\sum_{j=1}^{d}r_j^2
+
O(R^4).
\]
Finally,
\[
\|S_{\sigma,\boldsymbol{r}}\|
= \|T_{\sigma,\boldsymbol{r}}\|^2=
\left\|C_{\sigma+\sum_{j=1}^{d}r_jp_j^{-s}}\right\|^{2}.
\]
This proves the theorem.
\end{proof}

\section{Finite section error estimate}
\label{sec:5}
In this section, we develop fully finite-dimensional approximations with explicit total-degree and Dirichlet-sum truncation errors.

We introduce the positive coefficient operator associated with \(T_{\sigma, \boldsymbol{r}}\), obtained by removing the alternating signs
\[
(T_{\sigma,\boldsymbol{r}}^{+}a)_{\alpha}
=
\sum_{n=1}^{\infty}
a_n n^{-\sigma}
\frac{(\log n)^{|\alpha|}\boldsymbol{r}^{\alpha}}{\alpha!}.
\]
Define the diagonal unitary operator on $\ell^2(\mathbb{N}_0^d)$ by
\[
Ue_\alpha=(-1)^{|\alpha|}e_\alpha.
\]
Then
\[
T^+_{\sigma,\boldsymbol r}
=UT_{\sigma,\boldsymbol r},
\qquad
\|T^+_{\sigma,\boldsymbol r}\|
=
\|T_{\sigma,\boldsymbol r}\|.
\]
Our approach is based on truncating the positive matrix $A_{\sigma,\boldsymbol{r}}^{+}
=T_{\sigma,\boldsymbol{r}}^{+}(T_{\sigma,\boldsymbol{r}}^{+})^{*}$ to finite sections and estimating the resulting approximation error. This gives fully rigorous finite matrix approximations of the norm.
For $k,N\in\mathbb{N}_{0},$ define
\[
m_k(\sigma)
=
\sum_{n=1}^{\infty}n^{-2\sigma}(\log n)^k = (-1)^k\zeta^{(k)}(2\sigma) \quad \text{ and } \quad a_N(\boldsymbol{r})^2 =
\sum_{|\alpha|=N}
\frac{\boldsymbol{r}^{2\alpha}}{(\alpha!)^2}.
\]

\begin{proposition}
\label{prop:finite-prime-total-degree-reduction}
Fix $\varphi=\sigma+\sum_{j=1}^{d}r_jp_j^{-s}$  satisfying  \eqref{eq:fixed-sigma}.
Then
$A_{\sigma,\boldsymbol{r}}^{+}$
is unitarily equivalent to
$B_{\sigma,\boldsymbol{r}}\oplus \mathbf{0},$ where
\[
B_{\sigma,\boldsymbol{r}}
=
\left(
m_{N+M}(\sigma)a_N(\boldsymbol{r})a_M(\boldsymbol{r})
\right)_{N,M\geq 0}.
\]
Consequently, $\|C_{\varphi}\|^{2}
= \|B_{\sigma,\boldsymbol{r}}\|.$

\end{proposition}
\begin{proof}
For $ \alpha\in\mathbb{N}_{0}^{d}$, set  $w_{\alpha} =\frac{\boldsymbol{r}^{\alpha}}{\alpha!}.$ Using the same computation as in Proposition \ref{prop:fp-output-kernel}, we get
\[
(A_{\sigma,\boldsymbol{r}}^{+})_{\alpha,\beta}
=
m_{|\alpha|+|\beta|}(\sigma)w_{\alpha}w_{\beta}.
\]
If $R=0$, then $B_{\sigma,\mathbf 0}$ has rank one, and the assertion is
immediate. Assume from now on that \(R>0\). Then, for every
$N\in\mathbb N_0$, we have $a_N(\boldsymbol r)>0$.
Define $\omega_N\in\ell^{2}(\mathbb{N}_{0}^{d})$
by
\[
(\omega_N)_{\alpha}
=
\begin{cases}
\displaystyle \frac{w_{\alpha}}{a_N(\boldsymbol{r})}, & |\alpha|=N,\\[6pt]
0, & |\alpha|\neq N.
\end{cases}
\]
Thus
\[
\|\omega_N\|^2
=
\frac{1}{a_N(\boldsymbol{r})^2}
\sum_{|\alpha|=N}w_{\alpha}^2
=
1,
\]
and the vectors \(\omega_N\) have disjoint supports for different \(N\). Hence $(\omega_N)_{N\geq 0}$ is an orthonormal sequence.

We first compute  on finitely supported vectors. For
$|\alpha|=N,$ we have
\[
(A_{\sigma,\boldsymbol{r}}^{+}\omega_M)_{\alpha}
=
\sum_{|\beta|=M}
m_{N+M}(\sigma)w_{\alpha}w_{\beta}
\frac{w_{\beta}}{a_M(\boldsymbol{r})}.
\]
Therefore
\[
(A_{\sigma,\boldsymbol{r}}^{+}\omega_M)_{\alpha}
=
m_{N+M}(\sigma)w_{\alpha}
\frac{1}{a_M(\boldsymbol{r})}
\sum_{|\beta|=M}w_{\beta}^{2}.
\]
Since $\sum_{|\beta|=M}w_{\beta}^{2} = a_M(\boldsymbol{r})^{2},$
we obtain
\[
(A_{\sigma,\boldsymbol{r}}^{+}\omega_M)_{\alpha}
=
m_{N+M}(\sigma)w_{\alpha}a_M(\boldsymbol{r}).
\]
On the layer $|\alpha|=N,$ we have
$w_{\alpha}=a_N(\boldsymbol{r})(\omega_N)_{\alpha}.$ Thus
\[
A_{\sigma,\boldsymbol{r}}^{+}\omega_M
=
\sum_{N=0}^{\infty}
m_{N+M}(\sigma)a_N(\boldsymbol{r})a_M(\boldsymbol{r})\omega_N.
\]
This identifies the restriction of \(A_{\sigma,\boldsymbol{r}}^{+}\) to
$\overline{\operatorname{span}}\{\omega_N:N\geq 0\}$ with the matrix $B_{\sigma,\boldsymbol{r}}.$

It remains to show that \(A_{\sigma,\boldsymbol{r}}^{+}\) vanishes on the orthogonal complement. Let \(x\) be finitely supported and orthogonal to every \(\omega_M\). Then
\[
0=\langle x,\omega_M\rangle
=
\frac{1}{a_M(\boldsymbol{r})}
\sum_{|\beta|=M}x_{\beta}w_{\beta}
\]
for every \(M\). Hence
\[
\sum_{|\beta|=M}x_{\beta}w_{\beta}=0
\]
for every \(M\). For \(|\alpha|=N\),
\[
(A_{\sigma,\boldsymbol{r}}^{+}x)_{\alpha}
=
\sum_{M=0}^{\infty}
m_{N+M}(\sigma)w_{\alpha}
\sum_{|\beta|=M}w_{\beta}x_{\beta}
=
0.
\]
Thus $A_{\sigma,\boldsymbol{r}}^{+}x=0$
for every finitely supported vector in the orthogonal complement.
Since finitely supported vectors in the orthogonal complement are dense
there and \(A_{\sigma,\boldsymbol r}^{+}\) is bounded, the operator
vanishes on the entire orthogonal complement.
Consequently, $A_{\sigma,\boldsymbol{r}}^{+}$
is unitarily equivalent to
$B_{\sigma,\boldsymbol{r}}\oplus 0$ and
\[
\|C_{\varphi}\|^{2}
=\|A_{\sigma,\boldsymbol{r}}^{+}\|=\|B_{\sigma,\boldsymbol{r}}\|. \]
\end{proof}
To construct finite-dimensional approximations of the matrix $B_{\sigma,\boldsymbol{r}}$, we first define the truncation of $m_k(\sigma),$ given by
\[
m_k^{(M)}(\sigma)
=
\sum_{n=1}^{M}
n^{-2\sigma}(\log n)^k.
\]
By Proposition~\ref{prop:finite-prime-total-degree-reduction}, $\|C_{\varphi}\|^2 = \|B_{\sigma,\boldsymbol{r}}\|$. We therefore introduce the  finite-section of $B_{\sigma,\boldsymbol{r}}$.
For $N\in\mathbb{N}$,
\[
B_{\sigma,\boldsymbol{r}}^{(N)}
=
\left(
m_{j+k}(\sigma)a_j(\boldsymbol{r})a_k(\boldsymbol{r})
\right)_{0\le j,k\le N-1},
\]
and let
\[
\lambda_N(\sigma,\boldsymbol{r})
=
\lambda_{\max}(B_{\sigma,\boldsymbol{r}}^{(N)}).
\]
For numerical computation it is convenient to truncate the $m_k(\sigma)$ as well. This leads to the following fully finite approximation.
For $N,M\in\mathbb N$, define
\[
B_{\sigma,\boldsymbol{r}}^{(N,M)}
=
\left(
m_{j+k}^{(M)}(\sigma)a_j(\boldsymbol{r})a_k(\boldsymbol{r})
\right)_{0\le j,k\le N-1},
\]
and let
\[
\lambda_{N,M}(\sigma,\boldsymbol{r})
=
\lambda_{\max}(B_{\sigma,\boldsymbol{r}}^{(N,M)}).
\]
Let $P_N:\ell^2(\mathbb N_0)\to\ell^2(\mathbb N_0)$ be the orthogonal projection onto
$\operatorname{span}\{e_0,\ldots,e_{N-1}\}.$ Then
\[
B_{\sigma,\boldsymbol{r}}^{(N)} = P_NB_{\sigma,\boldsymbol{r}}P_N.
\]

\begin{lemma}
\label{lem:finite-section-lower-bound}
Let $\varphi=\sigma+\sum_{j=1}^{d}r_jp_j^{-s}$   satisfying  \eqref{eq:fixed-sigma}. Then, for every $N,M\in\mathbb{N},$ we have
\[
0\leq
\lambda_{N,M}(\sigma,\boldsymbol{r})
\leq
\lambda_N(\sigma,\boldsymbol{r})
\leq
\|C_{\varphi}\|^{2}.
\]
\end{lemma}
\begin{proof}
Since $B_{\sigma,\boldsymbol{r}}$ is Hilbert-space positive, its compression
$B_{\sigma,\boldsymbol{r}}^{(N)}
= P_NB_{\sigma,\boldsymbol{r}}P_N$
is a positive selfadjoint matrix. Hence
\[
\lambda_N(\sigma,\boldsymbol{r})
=
\|B_{\sigma,\boldsymbol{r}}^{(N)}\|
\geq 0.
\]
Moreover,
\[
\|B_{\sigma,\boldsymbol{r}}^{(N)}\|
=
\|P_NB_{\sigma,\boldsymbol{r}}P_N\|
\leq
\|B_{\sigma,\boldsymbol{r}}\|
=
\|C_{\varphi}\|^{2}.
\]
Now fix $N,M\in\mathbb{N}.$ For $n\in\mathbb{N},$ define
$v_n^{(N)}(\sigma,\boldsymbol{r})\in\mathbb{C}^{N}$
by
\[
(v_n^{(N)}(\sigma,\boldsymbol{r}))_j
=
n^{-\sigma}a_j(\boldsymbol{r})(\log n)^j,
\qquad
0\leq j\leq N-1.
\]
A direct computation gives
\[
B_{\sigma,\boldsymbol{r}}^{(N)}
=
\sum_{n=1}^{\infty}
v_n^{(N)}(\sigma,\boldsymbol{r})\otimes v_n^{(N)}(\sigma,\boldsymbol{r})
\]
and
\[
B_{\sigma,\boldsymbol{r}}^{(N,M)}
=
\sum_{n=1}^{M}
v_n^{(N)}(\sigma,\boldsymbol{r})\otimes v_n^{(N)}(\sigma,\boldsymbol{r}).
\]
Since each rank-one term is positive semidefinite,
\[
0\leq
B_{\sigma,\boldsymbol{r}}^{(N,M)}
\leq
B_{\sigma,\boldsymbol{r}}^{(N)}.
\] Hence
\[
0\leq
\lambda_{N,M}(\sigma,\boldsymbol{r})
\leq
\lambda_N(\sigma,\boldsymbol{r}).
\]
\end{proof}

\begin{lemma}
\label{lem:finite-prime-total-degree-tail-domination}
Let $\varphi=\sigma+\sum_{j=1}^{d}r_jp_j^{-s}$   satisfying  \eqref{eq:fixed-sigma}. For $N\in\mathbb{N},$  $Q_N:\ell^{2}(\mathbb{N}_{0}^{d})\to\ell^{2}(\mathbb{N}_{0}^{d})$
denotes the projection onto the multi-indices of total degree less than \(N\):
\[
Q_N(\ell^{2}(\mathbb{N}_{0}^{d}))
=
\overline{\operatorname{span}}
\{e_{\alpha}:|\alpha|\leq N-1\}.
\]
Then,
\[
0\leq
\|C_{\varphi}\|^{2}
-
\lambda_N(\sigma,\boldsymbol{r})
\leq
\left\|
(T_{\sigma,\boldsymbol{r}}^{+})^{*}(I-Q_N)T_{\sigma,\boldsymbol{r}}^{+}
\right\|.
\]
\end{lemma}
\begin{proof}
For every $x\in\ell^{2}(\mathbb{N}),$ the orthogonal decomposition
\[
T_{\sigma,\boldsymbol{r}}^{+}x
=
Q_NT_{\sigma,\boldsymbol{r}}^{+}x
+
(I-Q_N)T_{\sigma,\boldsymbol{r}}^{+}x
\]
gives
\[
\|T_{\sigma,\boldsymbol{r}}^{+}x\|^{2}
=
\|Q_NT_{\sigma,\boldsymbol{r}}^{+}x\|^{2}
+
\|(I-Q_N)T_{\sigma,\boldsymbol{r}}^{+}x\|^{2}.
\]
Taking the supremum over $\|x\|=1$, gives
\[
\|T_{\sigma,\boldsymbol{r}}^{+}\|^{2}
\leq
\|Q_NT_{\sigma,\boldsymbol{r}}^{+}\|^{2}
+
\|(I-Q_N)T_{\sigma,\boldsymbol{r}}^{+}\|^{2}.
\]
The total-degree reduction identifies
$\|Q_NT_{\sigma,\boldsymbol{r}}^{+}\|^{2}$ with $\lambda_N(\sigma,\boldsymbol{r}).$ Furthermore,
\[
\|(I-Q_N)T_{\sigma,\boldsymbol{r}}^{+}\|^{2}
=
\left\|
(T_{\sigma,\boldsymbol{r}}^{+})^{*}(I-Q_N)T_{\sigma,\boldsymbol{r}}^{+}
\right\|.
\]
Since $T_{\sigma,\boldsymbol{r}}^{+}$ is unitarily equivalent to  $T_{\sigma,\boldsymbol{r}}$,
$\|C_{\varphi}\|^{2}
= \|T_{\sigma,\boldsymbol{r}}^{+}\|^{2}.$
Thus
\[
\|C_{\varphi}\|^{2}
-
\lambda_N(\sigma,\boldsymbol{r})
\leq
\left\|
(T_{\sigma,\boldsymbol{r}}^{+})^{*}(I-Q_N)T_{\sigma,\boldsymbol{r}}^{+}
\right\|.
\]
\end{proof}
\begin{lemma}
\label{lem:finite-prime-output-tail-rank-one-expansion}
Fix $\sigma,\boldsymbol{r}$ satisfying \eqref{eq:fixed-sigma}. Then, we have
\[
(T_{\sigma,\boldsymbol{r}}^{+})^{*}(I-Q_N)T_{\sigma,\boldsymbol{r}}^{+}
=
\sum_{|\alpha|\geq N}
\boldsymbol{r}^{2\alpha}(u_{\alpha}\otimes u_{\alpha}),
\]
where the series converges absolutely in operator norm.
\end{lemma}
\begin{proof}
Recall that for $\alpha\in\mathbb{N}_{0}^{d},$
\[
(u_{\alpha})_n
=
n^{-\sigma}
\frac{(\log n)^{|\alpha|}}{\alpha!},
\qquad
n\in\mathbb{N}.
\]
For finitely supported \(a\), the \(\alpha\)-coordinate of
$T_{\sigma,\boldsymbol{r}}^{+}a$ is
\[
(T_{\sigma,\boldsymbol{r}}^{+}a)_{\alpha}
=
\boldsymbol{r}^{\alpha}\langle a,u_{\alpha}\rangle.
\]
The projection $I-Q_N$ keeps exactly those coordinates with
$|\alpha|\geq N.$
Therefore, for finitely supported \(a,c\in\ell^{2}(\mathbb{N})\),
\[
\left\langle
(T_{\sigma,\boldsymbol{r}}^{+})^{*}(I-Q_N)T_{\sigma,\boldsymbol{r}}^{+}a,c
\right\rangle
=
\sum_{|\alpha|\geq N}
\boldsymbol{r}^{2\alpha}
\langle a,u_{\alpha}\rangle
\langle u_{\alpha},c\rangle= \sum_{|\alpha|\geq N}
\langle \boldsymbol{r}^{2\alpha}(u_\alpha\otimes u_\alpha)a,c \rangle.
\]
Thus
\[
(T_{\sigma,\boldsymbol{r}}^{+})^{*}(I-Q_N)T_{\sigma,\boldsymbol{r}}^{+}
=
\sum_{|\alpha|\geq N}
\boldsymbol{r}^{2\alpha}(u_{\alpha}\otimes u_{\alpha})
\]
on finitely supported vectors and the convergence assertion follows from the Lemma \ref{lem:finite-prime-input-rank-one-expansion}.
\end{proof}

\begin{theorem}
\label{thm:finite-section}
Fix $\varphi=\sigma+\sum_{j=1}^{d}r_jp_j^{-s}$ satisfying  \eqref{eq:fixed-sigma}.
If $2R<\varepsilon<2\sigma-1$ then for $N\in\mathbb{N},$
\[
0\leq
\|C_{\varphi}\|^{2}
-
\lambda_N(\sigma,\boldsymbol{r})
\leq
E_N(\sigma,\boldsymbol{r},\varepsilon),
\]
where
\[
E_N(\sigma,\boldsymbol{r},\varepsilon)
=
\zeta(2\sigma-\varepsilon)
\frac{(4R^{2}/\varepsilon^{2})^{N}}{1-4R^{2}/\varepsilon^{2}}.
\]
\end{theorem}

\begin{proof}

If \(R=0\), then \(B_{\sigma,\boldsymbol0}\) is rank one and every
finite section with \(N\geq1\) is exact. Hence the assertion is immediate.
Assume from now on that \(R>0\).
By Lemma \ref{lem:finite-prime-total-degree-tail-domination} and \ref{lem:finite-prime-output-tail-rank-one-expansion},
\[
0\leq
\|C_{\varphi}\|^{2}
-
\lambda_N(\sigma,\boldsymbol{r})
\leq
\left\|
\sum_{|\alpha|\geq N}
\boldsymbol{r}^{2\alpha}(u_{\alpha}\otimes u_{\alpha})
\right\|.
\]
Thus
\[
0\leq
\|C_{\varphi}\|^{2}
-
\lambda_N(\sigma,\boldsymbol{r})
\leq
\sum_{m=N}^{\infty}
\sum_{|\alpha|=m}
\boldsymbol{r}^{2\alpha}\|u_{\alpha}\|^{2}.
\]
From equation \eqref{eq: summation bound}, we have
\[
\sum_{|\alpha|=m}
\boldsymbol{r}^{2\alpha}\|u_{\alpha}\|^{2}
\leq
\zeta(2\sigma-\varepsilon)
\left(\frac{4R^2}{\varepsilon^2}\right)^m.
\]
Therefore
\[
0\leq
\|C_{\varphi}\|^{2}
-
\lambda_N(\sigma,\boldsymbol{r})
\leq
\zeta(2\sigma-\varepsilon)
\sum_{m=N}^{\infty}
\left(\frac{4R^2}{\varepsilon^2}\right)^m.
\]
Since $2R<\varepsilon,$ we have
$0\leq \frac{4R^2}{\varepsilon^2}<1.$ Thus the geometric sum gives
\[
\sum_{m=N}^{\infty}
\left(\frac{4R^2}{\varepsilon^2}\right)^m
=
\frac{(4R^{2}/\varepsilon^{2})^{N}}{1-4R^{2}/\varepsilon^{2}}.
\]
This proves the theorem.
\end{proof}
The preceding theorem shows that $\|C_\varphi\|^2$ admits
finite-dimensional approximations with an explicit error bound. However, the entries of these finite matrices still
involve infinite Dirichlet-series sums. We now turn to the truncation of
these sums. The following lemma provides the estimate needed for this
additional approximation.
\begin{lemma}
\label{lem:finite-prime-universal}
Fix $\sigma,\boldsymbol{r}$ satisfying \eqref{eq:fixed-sigma}.
For $N,M\in\mathbb{N}$ and
$2R<\varepsilon<2\sigma-1,$ we have
\[
0\leq
\lambda_N(\sigma,\boldsymbol{r})
-
\lambda_{N,M}(\sigma,\boldsymbol{r})
\leq
\eta_{N,M}(\sigma,\boldsymbol{r},\varepsilon),
\]
where
\[
\eta_{N,M}(\sigma,\boldsymbol{r},\varepsilon)
=
\left(
\sum_{j=0}^{N-1}
\frac{(2j)!}{\varepsilon^{2j}}a_j(\boldsymbol{r})^{2}
\right)
\frac{M^{-(2\sigma-\varepsilon-1)}}{2\sigma-\varepsilon-1}.
\]
\end{lemma}

\begin{proof}
From Lemma \ref{lem:finite-section-lower-bound}, we have
\[
0\leq
\lambda_N(\sigma,\boldsymbol{r})
-
\lambda_{N,M}(\sigma,\boldsymbol{r}).
\]
Also,
\[
\lambda_N(\sigma,\boldsymbol{r})
-
\lambda_{N,M}(\sigma,\boldsymbol{r})
\leq
\left\|
B_{\sigma,\boldsymbol{r}}^{(N)}
-
B_{\sigma,\boldsymbol{r}}^{(N,M)}
\right\|.
\]
Recall $B_{\sigma,\boldsymbol{r}}^{(N)}$ and $
B_{\sigma,\boldsymbol{r}}^{(N,M)}$ from Lemma \ref{lem:finite-section-lower-bound}, then its difference is
\[
B_{\sigma,\boldsymbol{r}}^{(N)}
-
B_{\sigma,\boldsymbol{r}}^{(N,M)}
=
\sum_{n=M+1}^{\infty}
v_n^{(N)}(\sigma,\boldsymbol{r})\otimes v_n^{(N)}(\sigma,\boldsymbol{r}).
\]
Thus
\[
\left\|
B_{\sigma,\boldsymbol{r}}^{(N)}
-
B_{\sigma,\boldsymbol{r}}^{(N,M)}
\right\|
\leq
\sum_{n=M+1}^{\infty}
\|v_n^{(N)}(\sigma,\boldsymbol{r})\|_{\mathbb{C}^{N}}^{2}.
\]
Now calculate the norm of  $v_n^{(N)}(\sigma,\boldsymbol{r})$
\[
\|v_n^{(N)}(\sigma,\boldsymbol{r})\|_{\mathbb{C}^{N}}^{2}
=
n^{-2\sigma}
\sum_{j=0}^{N-1}
a_j(\boldsymbol{r})^{2}(\log n)^{2j}.
\]
Therefore
\[
\left\|
B_{\sigma,\boldsymbol{r}}^{(N)}
-
B_{\sigma,\boldsymbol{r}}^{(N,M)}
\right\|
\leq
\sum_{j=0}^{N-1}
a_j(\boldsymbol{r})^{2}
\sum_{n=M+1}^{\infty}
n^{-2\sigma}(\log n)^{2j}.
\]
From Equation \eqref{eq:exp-inequality}, we have
\[
(\log n)^{2j}
\leq
\frac{(2j)!}{\varepsilon^{2j}}n^{\varepsilon}.
\]
Hence
\[
\sum_{n=M+1}^{\infty}
n^{-2\sigma}(\log n)^{2j}
\leq
\frac{(2j)!}{\varepsilon^{2j}}
\sum_{n=M+1}^{\infty}
n^{-(2\sigma-\varepsilon)}.
\]
Since $2\sigma-\varepsilon>1,$ the tail estimate for zeta function gives
\[
\sum_{n=M+1}^{\infty}
n^{-(2\sigma-\varepsilon)}
\leq
\frac{M^{-(2\sigma-\varepsilon-1)}}{2\sigma-\varepsilon-1}.
\]
Combining these estimates yields
\[
\lambda_N(\sigma,\boldsymbol{r})
-
\lambda_{N,M}(\sigma,\boldsymbol{r})
\leq
\left(
\sum_{j=0}^{N-1}
\frac{(2j)!}{\varepsilon^{2j}}a_j(\boldsymbol{r})^{2}
\right)
\frac{M^{-(2\sigma-\varepsilon-1)}}{2\sigma-\varepsilon-1}.
\]

\end{proof}

\begin{theorem}\label{thm:full-error-estimate}
Fix $\varphi=\sigma+\sum_{j=1}^{d}r_jp_j^{-s}$ satisfying  \eqref{eq:fixed-sigma}.
If $2R<\varepsilon<2\sigma-1,$ then
for $N,M\in\mathbb{N}$
\[
0\leq
\|C_{\varphi}\|^{2}
-
\lambda_{N,M}(\sigma,\boldsymbol{r})
\leq
E_N(\sigma,\boldsymbol{r},\varepsilon)
+
\eta_{N,M}(\sigma,\boldsymbol{r},\varepsilon),
\]
where
\[
E_N(\sigma,\boldsymbol{r},\varepsilon)
=
\zeta(2\sigma-\varepsilon)
\frac{(4R^{2}/\varepsilon^{2})^{N}}{1-4R^{2}/\varepsilon^{2}},
\] and
\[
\eta_{N,M}(\sigma,\boldsymbol{r},\varepsilon)
=
\left(
\sum_{j=0}^{N-1}
\frac{(2j)!}{\varepsilon^{2j}}a_j(\boldsymbol{r})^{2}
\right)
\frac{M^{-(2\sigma-\varepsilon-1)}}{2\sigma-\varepsilon-1}.
\]
\end{theorem}
\begin{proof}
From Theorem \ref{thm:finite-section}, we have
\[
0\leq
\|C_{\varphi}\|^{2}
-
\lambda_N(\sigma,\boldsymbol{r})
\leq
E_N(\sigma,\boldsymbol{r},\varepsilon).
\]
And by Lemma \ref{lem:finite-prime-universal},
\[
0\leq
\lambda_N(\sigma,\boldsymbol{r})
-
\lambda_{N,M}(\sigma,\boldsymbol{r})
\leq
\eta_{N,M}(\sigma,\boldsymbol{r},\varepsilon).
\]
Adding the two inequalities gives
\[
0\leq
\|C_{\varphi}\|^{2}
-
\lambda_{N,M}(\sigma,\boldsymbol{r})
\leq
E_N(\sigma,\boldsymbol{r},\varepsilon)
+
\eta_{N,M}(\sigma,\boldsymbol{r},\varepsilon).
\]
\end{proof}

\vspace{0.2in}

\noindent\textbf{Data Availability:} No data were used or generated in this theoretical study.

\vspace{0.2in}
\noindent\textbf{Declaration of competing interests:}
The authors declare no competing interests.

\vspace{0.2in}
\noindent\textbf{Acknowledgments:}
X. Fang was partially supported by NSTC, Taiwan (No. 114-2115-M-A49-003-MY3).

\printbibliography
\end{document}